\PassOptionsToPackage{dvipsnames}{xcolor}
\documentclass[11pt]{amsart}

\usepackage[utf8]{inputenc}

\usepackage[margin=1in]{geometry}

\usepackage{amsmath, amsthm, amssymb}
\usepackage{mathtools}
\usepackage{mathrsfs}
\usepackage{stmaryrd}
\usepackage{stix}

\usepackage{amsmath,amssymb}
\usepackage{mathtools}
\usepackage{amsthm}
\usepackage{mathrsfs}
\usepackage{enumitem}
\usepackage{tikz}
\usetikzlibrary{calc}
\usepackage{nicefrac}
\usepackage{hyperref}

\urldef{\codeurl}\url{https://git-ce.rwth-aachen.de/Nicolas.espinoza.c/ls-nbpm-project/-/tree/v1.0-sisc}

\usepackage{siunitx}

\usepackage{graphicx}

\usepackage{tikz}
\newcommand*\circled[1]{\tikz[baseline=(char.base)]{
            \node[shape=circle,draw,inner sep=2pt] (char) {#1};}}

\usepackage{booktabs}
\usepackage{multirow}

\usepackage{float}
\usepackage{caption}
\usepackage{subcaption}

\usepackage{algorithm}
\usepackage{algorithmic}

\usepackage{tikz}
\usepackage{standalone}

\usepackage{calligra}
\usepackage{oubraces}

\usepackage{soul}   %

\usepackage{comment}

\usepackage[sort&compress,numbers,square]{natbib}
\usepackage{hyperref}
\hypersetup{
  colorlinks=true,
  linkcolor=Purple,
  citecolor=Orange,
  urlcolor=Purple
}

\theoremstyle{plain}
\newtheorem{theorem}{Theorem}[section]

\newtheorem{lemma}[theorem]{Lemma}

\theoremstyle{definition}

\newtheorem{Remark}[theorem]{Remark}

\newcommand\bx{\boldsymbol{x}}

\newcommand\bbeta{{\boldsymbol{u}}}
\newcommand\calP{{\mathcal{P}}}
\newcommand\calT{{\mathcal{T}}}

\newcommand\frakh{\mathfrak{h}}

\newcommand\bu{\boldsymbol{u}}
\newcommand\bv{\boldsymbol{v}}

\newcommand{\mesh}{\calP}

\newcommand{\bn}{{\boldsymbol{n}}}
\newcommand{\mean}[1]{\{\kern-1.1mm\{#1\}\kern-1.1mm\}}

\newcommand{\ndg}[1]{| \kern -.25mm \|{#1}| \kern -.25mm \|}
\newcommand{\ndgp}[1]{| \kern -.25mm \|{#1}| \kern -.25mm \|_{\rm \mesh}}
\newcommand{\ndgps}[1]{| \kern -.25mm \|{#1}| \kern -.25mm \|_{\rm \mesh,s}}
\newcommand{\R}{\mathbb{R}}
\newcommand{\N}{\mathbb{N}}

\newcommand{\dive}{\operatorname{div}}

\newcommand{\ih}{\ensuremath{i_h}} %
\newcommand{\Norm}[1]{\left\|#1\right\|}
\newcommand{\Ltwo}[1]{\left\|#1\right\|_{0,\,\Omega}}

\newcommand{\Lsnorm}[1]{\left\|#1\right\|_{\mathrm{LS}}}
\newcommand{\inner}[2]{\left\langle #1,\, #2 \right\rangle_{\Omega}}

\newcommand{\ncdg}[1]{| \kern -.25mm \|{#1}| \kern -.25mm \|_{\rm DG}}

\newcommand{\bnabla}{\nabla_{\kern-.05cm\mesh}^{}}

\title[NBPM and transport]{The nodally bound-preserving finite element method for the transport equation}

\author[N.~Espinoza-Contreras]{Nicol\'as Espinoza-Contreras}
\address[N.~Espinoza-Contreras]{Chair for Computational Analysis of Technical Systems, RWTH Aachen University, Aachen, Germany}
\email{espinoza@ssd.rwth-aachen.de}

\author[G.R.~Barrenechea]{Gabriel R. Barrenechea}
\address[G.R.~Barrenechea]{Department of Mathematics and
    Statistics, University of Strathclyde, 26 Richmond Street,
    Glasgow, G1 1XH United Kingdom}
\email{gabriel.barrenechea@strath.ac.uk}

\author[D.R.Q.~Pacheco]{Douglas R. Q. Pacheco}
\address[D.R.Q.~Pacheco]{Chair for Computational Analysis of Technical Systems, RWTH Aachen University, Aachen, Germany}
\email{pacheco@ssd.rwth-aachen.de}

\begin{document}
\maketitle

\noindent Keywords: bound-preserving FEM; transport equation; error estimates; Anderson Acceleration; variable density flow.

\begin{abstract}
Numerical approximations of transport problems, even when stabilised, often violate the bounds respected by the exact solution. Depending on the problem being solved, this can produce non-physical results, such as negative densities or higher-than-unity concentrations. Various bound-preserving numerical methods exist, yet they are usually either conditionally stable or only first-order accurate. In this context, the present work extends the nodally bound-preserving finite element method to the linear transport equation. As a baseline discretisation we use the  least-squares finite element method, which is instrumental in achieving unconditional stability for second-order time stepping. It also leads to symmetric algebraic systems, thereby enabling simple, efficient solvers, which we further enhance through Anderson acceleration. Unconditional stability is proved for the backward Euler and Crank--Nicolson discretisations. We also derive error estimates for the fully discrete second-order scheme and provide various numerical examples, including two-phase incompressible flows.
\end{abstract}

\section{Introduction}
\label{sec:intro}

The advection equation is the simplest of hyperbolic problems, which is why it is a common starting point to test new numerical methods. A thorough overview on the numerics of transport problems deserves a review work in its own right, so we simply quote here the reviews  \cite{LeVeque_2002} for finite volumes, \cite{Shu2016} for discontinuous Galerkin, and \cite{Ern-Guermond-III-2021} for finite element methods, all three in the context of more general hyperbolic problems.

One basic property of the (linear) transport equation is that the solution can be computed using the characteristics method. As a consequence, the values of the solution are bounded at any time by the extrema of the initial condition. Conserving this property in the discrete setting is far from immediate and has thus been an extensive research topic over the past few decades. Since Godunov's Theorem (see \cite{Godunov}), it is known that discretisations that preserve the bounds of the initial conditions must be nonlinear if higher orders of approximation are desired. This has given rise to numerous nonlinear discretisations that guarantee bound-preservation (or more stringent conditions, like Local Extrema Diminishing, or Total Variation Diminishing). We refer to \cite{ZS2011,KuzminH23,Ern-Guermond-III-2021} for very comprehensive surveys on the topic. 

The methods covered in those reviews are designed to satisfy stronger conditions than just bound preservation, while for many applications bound-preservation may be enough to yield stable, accurate approximations. A noticeable example is incompressible flow with variable density. There, the density comes from a linear transport equation, with convective term equal to the (solenoidal) flow velocity. In turn, the viscosity enters the Navier--Stokes momentum equation multiplying the density. So, maintaining a positive density is essential not only for physical consistency, but also for numerical stability, as a negative viscosity would destabilise the simulation. Moreover, in the seminal work \cite{Guermond2000}, the momentum equation was rewritten in a skew-symmetric form, where the square root of the density appears directly as a coefficient. So, maintaining a positive density is indispensable. As a matter of fact, in the analysis of virtually all works that followed, the positivity of the density is an essential---though sometimes overlooked---assumption \cite{Guermond2009,GueSal11,Salgado2013,Chen2020,Cai2021,Pan2024}.

Motivated, among other reasons, by the discussion above, in \cite{BGPV-2024} the {\it Nodally Bound-Preserving Method} (NBPM) was introduced as a way to hardwire the nodal bounds into the finite element method. Starting from a \textit{baseline discretisation}, that is, a stable discretisation of the PDE at hand, the method acts by defining a projection of the finite element unknown onto the closed convex set of admissible functions, i.e., the set of functions that satisfy the physical bounds at the degrees of freedom. Since this projection is not injective,  a kernel is introduced in the process and so a stabilising term is added to render the problem well-posed. To analyse the method, in \cite{BGPV-2024} it was reinterpreted as a variational inequality (see also \cite{Kirby-NumMath,Keith-FCOM} for other approaches using variational inequalities in the finite element discretisation). This method has also been extended to time-dependent problems in \cite{Amiri-Barrenechea-Pryer-2025} and to tensor-valued PDEs in \cite{ABP-26-M3AS}, both using the Continuous Interior Penalty (CIP) as baseline discretisation. Interestingly, the analysis of the time-dependent versions of this method has been restricted to the first-order implicit Euler method for time integration.

The purpose of this work is to extend the nodally bound-preserving method to the transport equation. The choice of baseline discretisation is the $\theta$-scheme for time stepping, and the Least-Squares (LS) finite element method (see, e.g., \cite{CaJi88,BoGu16}) for space discretisation. The method is presented and analysed in the stabilised form (i.e., as an equality), although some results are proven with the help of its equivalence with a variational inequality. This work advances the NBP FEM literature in the following aspects:
\begin{enumerate}
    \item  So far, only implicit Euler time-stepping has been analysed in combination with the NBPM. Here, we borrow techniques from, e.g., \cite{Pacheco2024}, to show stability of the method, and also derive error estimates for the Crank--Nicolson version. Crucially, neither of these results requires a CFL condition.%

\item We improve on the solution of the nonlinear system arising from the method at each time step. In \cite{BGPV-2024} a preconditioned Richardson iteration was used with good results, while in \cite{ABP-26-M3AS} a projected extragradient method was used to solve the variational inequality at each time step. In this work we have chosen to accelerate the fixed-point iterations using the Anderson algorithm (see \cite{Anderson-1965}). This has led to significant savings in the solution of the nonlinear systems and the overall discretisation as a consequence (see Section \ref{sec:Numer} for details and evidence on this). 

\item Finally, we have applied the current method to the approximation of incompressible flows with variable density. We have used the strategy proposed in \cite{Guermond2000} for the coupled problem together with a stable IMEX method that allows stably decoupling the momentum and density equations. The setup we have applied this methodology to is the Rayleigh--Taylor instability, where it is shown that our method substantially improves the quality of the numerical results, and it is even essential for the simulation to run. In fact, it is shown that for the baseline discretisation the density reaches a negative value in finite time, which makes the simulation crash, while the current method prevents this.
\end{enumerate}

The rest of the manuscript is organised as follows. After finishing the introduction with some basic notation and presentation of the model problem, in Section~\ref{sec:FEM} we present the current method. The method's stability is analysed in Section~\ref{Sec:stability} and the error analysis is carried out in Section~\ref{Sec:error}. Then, in Section~\ref{Sec:Anderson} we describe the Anderson acceleration, and apply it in the numerical solution of the nonlinear systems in the numerical experiments presented in Section~\ref{sec:Numer}. %

\subsection{Notations and the model problem}\label{sec:modelpb}
Let $\Omega\subseteq\R^d, d=1,2,3$ be an open, bounded,  polygonal/polyhedral, Lipschitz  domain, with outward
unit normal $\bn$.  We use standard notation for Sobolev spaces, aligned with, e.g.,  \cite{Ern-Guermond-I-2021}. In particular,
we denote by
$\|\cdot\|_{0,p,\Omega}$ the $L^p(\Omega)$-norm, and by $\|\cdot\|_{k,p,\Omega}$ the $W^{k,p}(\Omega)$-norm, omitting the subscript $p$ when
$p=2$. The inner product in $L^2(\Omega)$ will be denoted by $\langle\cdot,\cdot\rangle_\Omega^{}$, and we will not distinguish between inner products/norms of scalar and vector-valued functions.

In this work, we will analyse the following transport equation: find $\rho:(0,T]\times \Omega\to\R$ such that
\begin{align}
\partial_t\rho +\bbeta\cdot\nabla\rho &=\,0\qquad\textrm{in}\, (0,T]\times \Omega=: Q\,,\label{Transport:Eq}\\
\rho &=\, \rho_0^{}\qquad\textrm{in}\;\{0\}\times \Omega\,.\nonumber
\end{align}
We will assume that the convective velocity $\bbeta$ satisfies $\dive\bbeta=0$ in $Q$, and
that $\bbeta\cdot\bn=0$ on $(0,T]\times \partial\Omega$.  This last point justifies the absence of
boundary conditions on $\rho$. The initial condition $\rho_0^{}$ is assumed to be non-negative. So, defining $\kappa:= \textrm{ess sup}\,\{ \rho_0(\bx): \bx\in\Omega\}$, it is well known (see, e.g., \cite{Guermond2000}) that the method of characteristics can be used to show that  that the solution of \eqref{Transport:Eq} satisfies
\begin{equation}\label{Bounds-rho}
\rho\in [0,\kappa]\qquad\textrm{a.e.  in}\;\Omega,
\end{equation}
for all times $t > 0$.

\begin{Remark}
Although in the theoretical exposition we consider the interval $[0,\kappa]$, for simplicity, the NBP finite element method presented below accommodates more general bounds $[\kappa_1,\kappa_2]$ without any modifications. That is in fact the case in most of our numerical experiments. 
\end{Remark}

 \section{The discrete formulation}\label{sec:FEM}

Let $\{\mathcal{T}_h\}_{h>0}$  be a shape-regular family of simplicial triangulations of  $\Omega$.
Associated to each $\mathcal{T}_h$, and for $k\ge 1$,  we define the finite element space
\begin{equation}\label{Def:FE-spaces}
V_h:= \{ \phi_h\in C^0(\overline{\Omega}): \phi_h|_K\in \mathbb{P}_k(K)\, \forall\, K\in \calT_h\}\,.
\end{equation}
We denote by $\{\bx_1^{},\ldots,\bx_M^{}\}$ the set of  nodes of $\calT_h^{}$ (including the boundary nodes), and
by $\{\phi_1^{},\ldots,\phi_M^{}\}$ the associated Lagrangian basis.

The diameter of a set \(G \subset \mathbb{R}^d\) is denoted by \(h_G\), and the mesh size is defined as \(h = \max\{h_{K} : K \in \calT_h\}\). We also define the mesh function \(\mathfrak{h}\) as a continuous, element-wise linear function that represents a local average of element diameters. To construct this, we introduce the set of \textit{vertices} of the mesh, \(\boldsymbol{v}_1, \ldots, \boldsymbol{v}_I\), and define \(\mathfrak{h}\) as the piecewise linear function specified by the nodal values
\begin{equation}\label{mesh-function-definition}
	\frakh(\boldsymbol{v}_i^{})=\dfrac{\sum_{K: \boldsymbol{v}_i^{}\in K}h_K^{}}{\# \{K: \boldsymbol{v}_i^{}\in K\}}\,.
\end{equation}

Next, we recall that the Lagrange interpolation operator is defined by
(see, e.g., \cite[Chapter~11]{Ern-Guermond-I-2021})
\begin{align}
	i_{h}&:\mathcal{C}^{0}\big(\overline{\Omega}\big)\longrightarrow V_h,\nonumber\\
	v&\longmapsto i_{h}v=\sum_{j=1}^{M}v(\boldsymbol{x}_{j})\phi_{j}. \label{lagrange}
\end{align}
It is well-known (see, e.g., \cite[Proposition~11.12]{Ern-Guermond-I-2021}) that the Lagrange interpolant
satisfies the following approximation property:  Let  $1\leq \ell \leq k$ and
	$i_{h}$ be the Lagrange interpolant. Then, there exists $C>0$,
	independent of $h$,  such that for all $h$ and $\phi\in H^{\ell+1}(\Omega)$ the
	following holds:
	\begin{equation}
		\|\phi- i_{h}\phi \|_{0,K}+h_{K} |\phi-i_{h}\phi |_{1,K}\leq Ch^{\ell +1}_{K}|\phi |_{\ell +1,K}\qquad\forall K\in \calT_h\,.\label{lagranges}
	\end{equation}

For the time discretisation, we will consider a uniform time-step $\tau=T/N$, $N\in\N$.  We define $t^n=n\tau$ for
$n=0,1,\ldots,N$. The analysis presented
below also extends, without significant differences, to the case of non-uniform time stepping.

\subsection{The baseline discretisation: least-squares finite element methods}

The first ingredient in the formulation of the nodally bound-preserving method is the choice of a baseline discretisation. In this work we use the Least-Squares finite element method (see, e.g., \cite{BoGu16}) as baseline for space discretisation, while the time discretisation is  based on
the $\theta$-scheme. More precisely, the LS finite element method reads as follows: For a chosen $\theta\in\big[\tfrac12,1\big]$, find $\rho_h^{n+1}\in V_h^{}$ such that
\begin{equation}\label{LS-FEM}
\langle\rho_h^{n+1}-\rho_h^n + \tau\theta\, \bbeta^{n+1}\cdot\nabla\rho_h^{n+1}+ \tau(1-\theta)\, \bbeta^{n}\cdot\nabla\rho_h^{n}, \phi_h^{}+\tau\theta\, \bbeta^{n+1}\cdot\nabla\phi_h^{}\rangle_\Omega=0\,,
\end{equation}
for all $\phi_h^{}\in V_h^{}$. A stability analysis of this method can be found in \cite{Pacheco2024}.

\subsection{The  LS-NBP method}

The main motivation of the nodally bound-preserving method presented in \cite{BGPV-2024} is to build a finite element method where the condition  \eqref{Bounds-rho} is hardwired
into its definition. The next ingredient of the method is the definition of the set of  \textit{compliant} finite element functions, that is
\begin{equation}\label{Def:Vplus}
V_h^+:=\{v_h\in V_h^{}: v_h^{}(\bx_i^{})\in [0,\kappa]\;\textrm{for all}\; i=1,\ldots,M\}\,.
\end{equation}
Associated to this set, we define the projection $(\cdot)^+:V_h^{}\to V_h^+$ as follows
\begin{equation}\label{Def:plus-operator}
v_h^{}\longmapsto v_h^+:=\sum_{i=1}^M \max\{\min\{\kappa,v_h^{}(\bx_i^{})\}, 0\}\phi_i^{}\in V_h^+\,.
\end{equation}
Using this mapping,  every $v_h\in V_h^{}$ can be decomposed as
\begin{equation*}
v_h^{}= v_h^+ + (v_h^{}-v_h^+) =: v_h^+ +v_h^-\,.
\end{equation*}

With these notations, the LS-NBPM for \eqref{Transport:Eq} reads as follows: Start with $\rho_h^0=i_h(\rho_0^{})$, and for $n\in\{0,\ldots,N-1\}$, find $\rho_h^{n+1}\in V_h^{}$ such that
\begin{equation}\label{LS-NBPM}
\left\langle  (\rho_h^{n+1})^+ -(\rho_h^n)^+ + \tau\theta\, \bbeta^{n+1}\cdot\nabla(\rho_h^{n+1})^++ \tau(1-\theta)\, \bbeta^{n}\cdot\nabla(\rho_h^{n})^+, \phi_h^{}+\tau\theta\, \bbeta^{n+1}\cdot\nabla\phi_h^{}\right\rangle_\Omega
+ s( (\rho_h^{n+1})^-, \phi_h^{})=0\,,
\end{equation}
for all $\phi_h^{}\in V_h^{}$. Here, $s(\cdot,\cdot)$ is the following stabilising bilinear form
\begin{equation}\label{Def:stab-form}
s(\phi_h,\psi_h)=\alpha\sum_{i=1}^M \left(\frakh (\boldsymbol{x}_{i})^d+\|\bbeta\|_{\infty,\omega_i}\tau\,\frakh (\boldsymbol{x}_{i})^{d-1}\right)\phi_h^{}(\bx_i^{})\psi_h^{}(\bx_i^{})\,,
\end{equation}
where $\alpha>0$ is a non-dimensional constant and $\omega_i:=\cup\{ K\in\calT_h: \bx_i\in K\}$.

The bilinear form $s(\cdot,\cdot)$ satisfies the following properties (see  \cite[Lemma~3.1 and Eq.~(3.8)]{BGPV-2024}):
\begin{align}
s(\phi_h^--\psi_h^-,\phi_h^+-\psi_h^+) &\ge\, 0\,,\label{Eq:s-monotone-0}\\
s(\phi_h^-,\phi_h^+) &\ge\, 0\,,\label{Eq:s-monotone}\\
s(\phi_h^{},\phi_h^{}) &\ge\, C\,\|\phi_h^{}\|_{0,\Omega}^2\,,
\end{align}
for every $\phi_h^{}, \psi_h\in V_h^{}$, where $C >0$ is independent of $h$.

\begin{Remark}
   The problem \eqref{LS-NBPM} can be written in an equivalent way as follows:
Find $\rho_h^{n+1}\in V_h^{}$ such that
\begin{align}
&\underbrace{\left\langle  (\rho_h^{n+1})^+  + \tau\theta\, \bbeta^{n+1}\cdot\nabla(\rho_h^{n+1})^+, \phi_h^{}+\tau\theta\, \bbeta^{n+1}\cdot\nabla\phi_h^{}\right\rangle_\Omega}_{=: a^{n+1}((\rho_h^{n+1})^+,\phi_h^{})}
+ \ s( (\rho_h^{n+1})^-, \phi_h^{})\nonumber \\
&=\, \underbrace{\left\langle(\rho_h^n)^+ -  \tau(1-\theta)\, \bbeta^{n}\cdot\nabla(\rho_h^{n})^+, \phi_h^{}+\tau\theta\, \bbeta^{n+1}\cdot\nabla\phi_h^{}\right\rangle_\Omega^{}}_{=:L^{n+1}(\phi_h^{})}\,,\label{LS-NBPM-bis}
\end{align}
for all $\phi_h^{}\in V_h^{}$.  We observe that for every $n\in \{1,\ldots,N\}$ the bilinear form
$a^n(\cdot,\cdot)$ is symmetric and elliptic. In fact, integrating by parts and using that $\bbeta$ is solenoidal
and has vanishing normal component on $\partial\Omega$ we get the following ellipticity:
\begin{equation}\label{Eq:Exis-1}
a^{n}(\phi,\phi)=\|\phi\|_{0,\Omega}^2+\theta^2\tau^2\|\bbeta^{n}\cdot\nabla\phi\|_{0,\Omega}^2=: \|\phi\|_{\rm LS}^2\,,
\end{equation}
for every function regular enough so the norms above are finite. So, in particular, $a^n(\cdot,\cdot)$ is
elliptic in $V_h^{}$.
\end{Remark}

The writing of \eqref{LS-NBPM} done in the last remark makes it simpler to prove that the problem is well-posed. In fact, the following result states just that.

\begin{lemma}\label{Lem:Well-pos}
For each $n\in \{0,\ldots,N-1\}$ there exists a unique solution $\rho_h^{n+1}$ to \eqref{LS-NBPM}.
\end{lemma}

\begin{proof} 
We start by noticing that following arguments very similar  to those used in \cite[Section~3.2]{BGPV-2024} we deduce that
\eqref{LS-NBPM-bis} is equivalent to the following variational inequality: Find $\psi_h^{}\in V_h^+$ such
that
\begin{equation}\label{Eq:Exis-2}
a^{n+1}(\psi_h^{},\phi_h^{}-\psi_h^{})\ge L^{n+1}(\phi_h^{}-\psi_h^{})\qquad\forall\, \phi_h^{}\in V_h^{+}\,.
\end{equation}
(In fact, $\psi_h=(\rho_h^{n+1})^+$).
Since the bilinear form $a^{n+1}(\cdot,\cdot)$ is elliptic thanks to \eqref{Eq:Exis-1}, we can apply
Stampacchia's Theorem (see \cite[Chap. ~II,Theorem~2.1]{Stampacchia-2000}) and conclude that \eqref{Eq:Exis-2} is well-posed, which proves
the well-posedness of \eqref{LS-NBPM}, and the result.
\end{proof}

\section{Unconditional stability}\label{Sec:stability}

The stability for the method defined in \eqref{LS-NBPM} is the main topic of this section. The following result states that, for $\theta\in [1/2,1]$, \eqref{LS-NBPM} is unconditionally stable. 

\begin{theorem}\label{Theo:Stability}
Let $N\in \mathbb{N}$. Then, for every $n\in \{1,\ldots,N\}$ the following stability bound holds for the solution of \eqref{LS-NBPM}:
\begin{equation}\label{Eq:stab}
\left\|(\rho_h^n)^+\right\|^2_{0,\Omega}+ \left\|\tau(1-\theta)\bbeta^n\cdot\nabla (\rho_h^n)^+\right\|^2_{0,\Omega}
+ (2\theta-1)\tau^2\sum_{m=1}^n \left\|\bbeta^m\cdot\nabla (\rho_h^m)^+\right\|^2_{0,\Omega}
\le \left\|\rho_h^0\right\|^2_{0,\Omega}+\left \|\tau(1-\theta)\bbeta^0\cdot \nabla \rho_h^0\right\|^2_{0,\Omega}\,,
\end{equation}
for any time-step size $\tau = T/N$.
\end{theorem}

\begin{proof}
Let $n\in \{1,\ldots,N\}$ and $m\in \{1,\ldots,n\}$.
Then, taking $\phi_h^{}=2(\rho_h^{m})^+$ as test function in \eqref{LS-NBPM} we obtain:
\begin{align}
0 =&\, 2\left\langle (\rho_h^{m})^+-(\rho_h^{m-1})^+, (\rho_h^{m})^+\right\rangle_\Omega^{}
+2\left\langle\tau\theta\, \bbeta^m\cdot\nabla (\rho_h^{m})^+ +\tau(1-\theta)\bbeta^{m-1}\cdot\nabla (\rho_h^{m-1})^+,
(\rho_h^{m})^+\right\rangle_\Omega^{} \nonumber\\
&\,+ 2\left\langle(\rho_h^{m})^+-(\rho_h^{m-1})^+,  \tau\theta \bbeta^{m}\cdot\nabla(\rho_h^m)^+ \right\rangle_\Omega^{}
\nonumber\\
&\, + 2\left\langle \tau\theta\, \bbeta^m\cdot\nabla (\rho_h^{m})^+ +\tau(1-\theta)\bbeta^{m-1}\cdot\nabla (\rho_h^{m-1})^+,
\tau\theta \bbeta^{m}\cdot\nabla(\rho_h^m)^+ \right\rangle_\Omega^{} + 2s((\rho_h^m)^-,(\rho_h^m)^+) \nonumber\\
=:&\,  \circled{1} + \circled{2}+\ldots+\circled{5}\,. \label{Eq:Stab-1}
\end{align}

We deal with the terms above  one by one.  First, using \eqref{Eq:s-monotone} it follows
immediately that $\circled{5}\ge 0$. Next, the identity $2(a-b)a=a^2-b^2+(a-b)^2$ yields
\begin{equation}\label{Eq:Stab-2}
\circled{1}=\|(\rho_h^{m})^+\|^2_{0,\Omega}-\|(\rho_h^{m-1})^+\|^2_{0,\Omega} +\|(\rho_h^{m})^+-(\rho_h^{m-1})^+\|^2_{0,\Omega}\,.
\end{equation}
Moreover, using that $\bbeta$ is solenoidal and has vanishing normal component on $\partial\Omega$ we see that the convective bilinear form is skew-symmetric, and thus
\begin{equation}
\circled{2}= 2\left\langle\tau(1-\theta)\bbeta^{m-1}\cdot\nabla(\rho_h^{m-1})^+, (\rho_h^{m})^+-(\rho_h^{m-1})^+\right\rangle_\Omega^{}\,.\label{Eq:stab-3}
\end{equation}

Once again, applying $2(a+b)a=a^2-b^2+(a+b)^2$ we arrive at
\begin{align}
\circled{4}=&\, \|\tau\theta\bbeta^m\cdot\nabla (\rho_h^m)^+\|^2_{0,\Omega}-
\|\tau(1-\theta)\bbeta^{m-1}\cdot\nabla (\rho_h^{m-1})^+\|^2_{0,\Omega}  \nonumber\\
&\,+ \| \tau\theta\bbeta^m\cdot\nabla (\rho_h^m)^+ + \tau(1-\theta)\bbeta^{m-1}\cdot\nabla (\rho_h^{m-1})^+\|^2_{0,\Omega}\,.\label{Eq:stab-4}
\end{align}

Replacing the rewritings of \circled{1},...,\circled{5}  into \eqref{Eq:Stab-1} and re-arranging terms gives us
\begin{align}
&\, \|(\rho_h^{m})^+\|^2_{0,\Omega}-\|(\rho_h^{m-1})^+\|^2_{0,\Omega} +\|(\rho_h^{m})^+-(\rho_h^{m-1})^+\|^2_{0,\Omega} \nonumber\\
&\, + 2\left\langle \tau(1-\theta)\bbeta^{m-1}\cdot\nabla (\rho_h^{m-1})^+ + \tau\theta \bbeta^{m}\cdot\nabla (\rho_h^m)^+,
(\rho_h^m)^+ - (\rho_h^{m-1})^+\right\rangle_\Omega^{}\nonumber\\
&\, + \|\tau\theta\bbeta^m\cdot\nabla (\rho_h^m)^+\|^2_{0,\Omega}-
\|\tau(1-\theta)\bbeta^{m-1}\cdot\nabla (\rho_h^{m-1})^+\|^2_{0,\Omega} \nonumber\\
&\, + \| \tau\theta\bbeta^m\cdot\nabla (\rho_h^m)^+ + \tau(1-\theta)\bbeta^{m-1}\cdot\nabla (\rho_h^{m-1})^+\|^2_{0,\Omega}\nonumber\\
=&\, \circled{i}+\ldots+\circled{vii}\le 0\,.\label{Eq:stab-5}
\end{align}

We once again rewrite and bound the terms above one by one.  First, we see that
\begin{equation}\label{Eq:stab-6}
\circled{iii}+\circled{iv}+\circled{vii} =
\|(\rho_h^{m})^+-(\rho_h^{m-1})^+ +  \tau\theta\bbeta^m\cdot\nabla (\rho_h^m)^+ + \tau(1-\theta)\bbeta^{m-1}\cdot\nabla (\rho_h^{m-1})^+\|^2_{0,\Omega}\ge 0\,.
\end{equation}
Moreover,
\begin{equation}\label{Eq:stab-7}
\circled{v}= \theta^2\|\tau\bbeta^m\cdot\nabla (\rho_h^m)^+\|^2_{0,\Omega}
= (1-\theta)^2\|\tau\bbeta^m\cdot\nabla (\rho_h^m)^+\|^2_{0,\Omega}  - (1-2\theta)\|\tau\bbeta^m\cdot\nabla (\rho_h^m)^+\|^2_{0,\Omega}\,.
\end{equation}
Hence, \eqref{Eq:stab-5} implies
\begin{align}
&\, \|(\rho_h^{m})^+\|^2_{0,\Omega}-\|(\rho_h^{m-1})^+\|^2_{0,\Omega} +
\|\tau(1-\theta)\bbeta^m\cdot\nabla (\rho_h^m)^+\|^2_{0,\Omega}-
\|\tau(1-\theta)\bbeta^{m-1}\cdot\nabla (\rho_h^{m-1})^+\|^2_{0,\Omega} \nonumber\\
&\, -(1-2\theta)\, \|\tau\bbeta^m\cdot\nabla (\rho_h^m)^+\|^2_{0,\Omega}\le 0\,. \label{Eq:stab-8}
\end{align}
The proof then follows adding from $m=1$ to $m=n$, which gives us stability for $\theta \in [0.5,1]$. 
\end{proof}

\section{Error analysis}\label{Sec:error}
In this section, we derive a priori error estimates for \eqref{LS-NBPM}. We herein focus on the fully-discrete Crank--Nicolson scheme, for two reasons. First, the same arguments also apply to the full range of stable $\theta$-schemes. In addition, this is the first time that the analysis of a higher-order time discretisation has been carried out for the nodally bound-preserving finite element method. To avoid unnecessary technicalities, the convergence analysis will assume that the convective velocity $\boldsymbol{u}$ does not depend on time.

We start by stating two preliminary inequalities. Their proofs are classical, but we state them here in a condensed way for completeness.
\begin{lemma}\label{lem:taylor}
Let $t^{n+1/2}:=(t^n+t^{n+1})/2$. If $g\in H^2(t^n,t^{n+1};L^2(\Omega))$, then
\begin{align*}
\left\|\frac{g(t^n)+g(t^{n+1})}{2}-g(t^{n+1/2})\right\|_{0,\Omega}
\le \Big(\frac{\tau^3}{48}\Big)^{1/2}\left\|\partial_t^2 g\right\|_{L^2(t^n,t^{n+1};L^2(\Omega))},
\end{align*}
and if $g\in H^3(t^n,t^{n+1};L^2(\Omega))$, then
\begin{align*}
\left\|\frac{g(t^{n+1})-g(t^n)}{\tau}-\partial_t g(t^{n+1/2})\right\|_{0,\Omega}
\le \Big(\frac{\tau^3}{320}\Big)^{1/2}\left\|\partial_t^3 g\right\|_{L^2(t^n,t^{n+1};L^2(\Omega))}.
\end{align*}
\end{lemma}
\begin{proof}
Let $w(s):=\min\{s-t^n,\,t^{n+1}-s\}$. Taylor expansion about $t^{n+1/2}$ with
integral remainder (the odd-order terms cancel by symmetry) gives
\begin{align*}
\frac12\big(g(t^n)+g(t^{n+1})\big)-g(t^{n+1/2})=&\,\frac12\int_{t^n}^{t^{n+1}} w(s)\,\partial_t^2 g(s)\,\mathrm{d}s,
\\
\frac{g(t^{n+1})-g(t^n)}{\tau}-\partial_t g(t^{n+1/2})=&\,\frac1\tau\int_{t^n}^{t^{n+1}} \frac{w(s)^2}{2}\,\partial_t^3 g(s)\,\mathrm{d}s .
\end{align*}
The result follows from the Cauchy--Schwarz inequality, using
$\int_{t^n}^{t^{n+1}} w^2 = \tau^3/12$ and $\int_{t^n}^{t^{n+1}} w^4/4 = \tau^5/320$.
\end{proof}

We are now in position of proving our main error estimate.

\begin{theorem}\label{Theo:Convergence}
	Provided that the exact solution $\rho$ of the transport equation \eqref{Transport:Eq} is such that
	\begin{align*}
	\partial_t\rho \in L^2(0,T;H^{k+1}(\Omega))\, , \quad \partial^3_t\rho \in L^2(Q)\, , %
    \quad\text{and} \ \ \sum_{n=0}^N\tau\left|\rho^n\right|_{k+1,\,\Omega} < \infty\, ,
	\end{align*}
the following error bound holds for the solution of \eqref{LS-NBPM} with $\theta=1/2$:
\begin{align}
		\Lsnorm{\rho^N - \left(\rho_h^N\right)^+  }
	\leq  C_1 \tau^2 +  C_2 h^k + C_3 h^{k+1}\, ,
\end{align}
where
\begin{alignat*}{2}
    C_1 &= C \sqrt{T} \Norm{\partial_t^3\rho}_{0,\,Q} \ ,
	\quad && C_2 = C u_{\infty}\sum_{n=0}^N\tau\left|\rho^n\right|_{k+1,\,\Omega} ,\\
	C_3 &= C\sqrt{T}\Norm{\partial_t\rho}_{L^2\left(0,T;H^{k+1}(\Omega)\right)}+C\left|\rho^N\right|_{k+1,\,\Omega} \ ,
    \quad && u_{\infty} = \|\boldsymbol{u}\|_{\infty,\,\Omega}\,,
\end{alignat*}
and $C$ denotes dimensionless constants.

\end{theorem}

\begin{proof}
We start by recalling that $\rho^n$ denotes the solution $\rho$ evaluated at $t=t^n = n\tau$, while $\rho^{n+\frac12} := \rho(t^{n+\frac12})$, $t^{n+\frac12} := (n+\nicefrac{1}{2})\tau$ . Let us then split the error into discrete and interpolation parts:
\begin{align}
	\left(\rho^n_h\right)^+ - \rho^n  =	\underbrace{\left(\rho^n_h\right)^+ - \ih\rho^n}_{=:\, e_h^n} + \underbrace{\ih\rho^n - \rho^n}_{=:\, \eta^n}\, .
\end{align}

Using that $\rho(t,\bx)\in [0,\kappa]$  in $Q$, then $(i_h\rho^{n+1})^+=i_h\rho^{n+1}$ and $(i_h\rho^{n+1})^-=0$. As a consequence, the numerical and exact solutions satisfy, respectively, the following equations:
\begin{align}
	\inner{
		\left(\rho_h^{n+1}\right)^+ - \left(\rho_h^n\right)^+  +\tau\boldsymbol{u} \cdot \nabla \left(\frac{\left(\rho_h^{n+1}\right)^+ + \left(\rho_h^n\right)^+}{2} \right)
	}{
		\phi_h + \frac{\tau}{2} \boldsymbol{u} \cdot \nabla\phi_h
	} + s\left(\left(\rho_h^{n+1}\right)^-,\phi_h\right)&= 0\, \label{errorEq1}\\
\inner{\tau \partial_t\rho^{n+\frac{1}{2}}
	+ \tau\boldsymbol{u} \cdot \nabla\rho^{n+\frac{1}{2}}
}{
	\phi_h + \frac{\tau}{2} \boldsymbol{u} \cdot \nabla\phi_h} + s\left(\left(\ih\rho^{n+1}\right)^-,\phi_h\right) &= 0\, . \label{errorEq2}
\end{align}
Subtracting \eqref{errorEq2} from \eqref{errorEq1} yields
\begin{align}
	&\inner{
	\left(\rho_h^{n+1}\right)^+ - \left(\rho_h^n\right)^+ - \tau \partial_t\rho^{n+\frac{1}{2}}
		+ \tau\boldsymbol{u} \cdot \nabla \left(\frac{\left(\rho_h^{n+1}\right)^+ + \left(\rho_h^n\right)^+}{2} - \rho^{n+\frac{1}{2}}\right)
	}{
	\phi_h + \frac{\tau}{2} \boldsymbol{u} \cdot \nabla\phi_h
	} \label{error1} \\
&= s\left(\left(\ih\rho^{n+1}\right)^- - \left(\rho_h^{n+1}\right)^-,\phi_h\right)  .\nonumber
\end{align}
By setting $\phi_h = e_h^{n+1} = (\rho_h^{n+1})^+ - \ih{\rho^{n+1}}$ in \eqref{error1}, we start noticing that, thanks to \eqref{Eq:s-monotone-0}, the resulting stabilisation term satisfies
\begin{equation*}
	s\left(\left[\ih\rho^{n+1}\right]^- - \left(\rho_h^{n+1}\right)^-,(\rho_h^{n+1})^+ - \ih{\rho^{n+1}}\right) =    -\overbrace{s\left( \left(\rho_h^{n+1}\right)^- - \left(\ih\rho^{n+1}\right)^- ,\left(\rho_h^{n+1}\right)^+ - \left(\ih\rho^{n+1}\right)^+\right)}^{\geq 0 \ \text{(monotonicity of $s$)}}\leq 0\,.
\end{equation*}
 Thus, \eqref{error1} leads to the following error inequality
\begin{align}
	\inner{
	e_h^{n+1} - e_h^n
		+ \tau\boldsymbol{u} \cdot \nabla \left(\frac{e_h^{n+1} + e_h^n}{2} \right)
	}{
		e_h^{n+1} + \frac{\tau}{2} \boldsymbol{u} \cdot \nabla e_h^{n+1}
	} \leq 	\inner{\text{RHS}
}{
e_h^{n+1} + \frac{\tau}{2} \boldsymbol{u} \cdot \nabla e_h^{n+1}
}
\, ,\label{error2}
\end{align}
where 
\begin{align*}
\text{RHS} := 	\tau\partial_t\rho^{n+\frac12} - \ih{(\delta\rho^{n+1})}
	+ \tau\boldsymbol{u} \cdot \nabla \left[\rho^{n+\frac12} - \ih{\left(\frac{\rho^{n+1} + \rho^n}{2}\right)}\right]\,,
\end{align*}
and $\delta\rho^{n+1}:= \rho^{n+1}-\rho^n$.

The left-hand side of \eqref{error2} is equal to
\begin{align*}
\Ltwo{e_h^{n+1} + \frac{\tau}{2} \boldsymbol{u} \cdot \nabla e_h^{n+1}}^2 - 	\inner{
	e_h^n
	- \frac{\tau}{2}\boldsymbol{u} \cdot \nabla e_h^n
}{
	e_h^{n+1} + \frac{\tau}{2} \boldsymbol{u}\cdot \nabla e_h^{n+1}
}\, ,
\end{align*}
so that the Cauchy--Schwarz inequality gives us
\begin{align*}
	\Ltwo{e_h^{n+1} + \frac{\tau}{2} \boldsymbol{u} \cdot \nabla e_h^{n+1}}^2 \leq	\inner{
		e_h^n
		- \frac{\tau}{2}\boldsymbol{u} \cdot \nabla e_h^n + \text{RHS}
	}{
		e_h^{n+1} + \frac{\tau}{2}\boldsymbol{u} \cdot \nabla e_h^{n+1}
	}\\
\leq
\left(\Ltwo{e_h^n
- \frac{\tau}{2}\boldsymbol{u} \cdot \nabla e_h^n} + \Ltwo{\text{RHS}}\right)
\Ltwo{e_h^{n+1} + \frac{\tau}{2}\boldsymbol{u} \cdot \nabla e_h^{n+1}}\, ,
\end{align*}
from which we conclude
\begin{align}
	\Ltwo{e_h^{n+1} + \frac{\tau}{2}\boldsymbol{u} \cdot \nabla e_h^{n+1}}
	\leq
	\Ltwo{e_h^n
		- \frac{\tau}{2}\boldsymbol{u} \cdot \nabla e_h^n} + \Ltwo{\text{RHS}}\, .\label{error3}
\end{align}
Due to skew-symmetry, we have
\begin{equation*}
	\Ltwo{e_h^{n} \pm \frac{\tau}{2} \boldsymbol{u} \cdot \nabla e_h^{n}}^2
	= \Ltwo{e_h^{n}}^2 + \left(\frac{\pm \tau}{2}\right)^2\!  \Ltwo{\boldsymbol{u} \cdot \nabla e_h^{n}}^2
	\pm \tau \overbrace{\inner{\boldsymbol{u}\cdot \nabla e_h^{n}}{e_h^{n}}}^{=\, 0} 
	= \Norm{e_h^{n}}_{\text{LS}}^2\, ,
\end{equation*}
so that \eqref{error3} becomes
\begin{align*}
 \Lsnorm{e_h^{n+1}} \leq \Lsnorm{e_h^n} + \Ltwo{\text{RHS}}\, .
\end{align*}

The right-hand side term can be split into multiple contributions:
\begin{align*}
	\Ltwo{\text{RHS}} =&\,  \Ltwo{\tau\partial_t\rho^{n+\frac12} - \ih{(\delta\rho^{n+1})}
	+ \tau\boldsymbol{u} \cdot \nabla \left[\rho^{n+\frac12} - \ih{\left(\frac{\rho^{n+1} + \rho^n}{2}\right)}\right]}\\
 \leq&\, \underbrace{\tau \Ltwo{\partial_t\rho^{n+\frac12} - \frac{\delta\rho^{n+1}}{\tau}}}_{\circled{1}}
+ \underbrace{\tau \Ltwo{\boldsymbol{u}\cdot \nabla \left[\rho^{n+\frac{1}{2}} - \left(\frac{\rho^{n+1} + \rho^n}{2}\right)\right]}}_{\circled{2}} \\
&+ \underbrace{\frac{\tau}{2} \Ltwo{\boldsymbol{u} \cdot \nabla(\eta^{n+1}+\eta^{n})}}_{\circled{3}}
+ \underbrace{\Ltwo{\delta\eta^{n+1}}}_{\circled{4}}\, .
\end{align*}
The first term can be bounded through the central-difference estimate derived from Lemma \ref{lem:taylor}:
\begin{align*}
	\tau \Ltwo{\partial_t\rho^{n+\frac12} - \frac{\delta\rho^{n+1}}{\tau}} \leq \tau C\tau^{\frac32}\left(\int_{t^n}^{t^{n+1}}\Ltwo{\partial^3_{t}\rho}^2\mathrm{d}t\right)^{\frac12} = C\tau^{\frac52}\left(\int_{t^n}^{t^{n+1}}\Ltwo{\partial^3_{t}\rho}^2\mathrm{d}t\right)^{\frac12}\, .
\end{align*}
For $\circled{2}$, since $\rho$ satisfies the PDE \eqref{Transport:Eq} at all times, we can write
\begin{align*}
	\boldsymbol{u} \cdot \nabla \left[\rho^{n+\frac{1}{2}} - \left(\frac{\rho^{n+1} + \rho^n}{2}\right)\right] =  -\partial_t\rho^{n+\frac{1}{2}} + \left(\frac{\partial_t\rho^{n+1} + \partial_t\rho^n}{2}\right).
\end{align*}
Therefore, through the midpoint estimate derived from Lemma \ref{lem:taylor}, we get
\begin{align*}
\tau\Ltwo{	\boldsymbol{u} \cdot \nabla \left[\rho^{n+\frac{1}{2}} - \left(\frac{\rho^{n+1} + \rho^n}{2}\right)\right]} &= \tau\Ltwo{\partial_t\rho^{n+\frac{1}{2}} - \left(\frac{\partial_t\rho^{n+1} + \partial_t\rho^n}{2}\right)} \\
&\leq C\tau^{\frac52}\left(\int_{t^n}^{t^{n+1}}\Ltwo{\partial^3_{t}\rho}^2\mathrm{d}t\right)^{\frac12},
\end{align*}
as in the previous term.

For $\circled{3}$ we use Hölder's inequality and the approximation property of $\ih$:
\begin{equation*}
	\frac{\tau}{2}\Ltwo{\boldsymbol{u}\cdot\nabla(\eta^{n+1}+\eta^{n})} \leq \frac{\tau}{2}u_{\infty}\left|\eta^{n+1}+\eta^n\right|_{1,\,\Omega}
	\leq \frac{\tau}{2}u_{\infty} C  h^k \left|\rho^{n+1}+\rho^n\right|_{k+1, \,\Omega} .
\end{equation*}
Similarly, $\circled{4}$ can be bounded using interpolation and the Cauchy--Schwarz inequality:
\begin{align*}
	\Ltwo{\delta\eta^{n+1}} &= \Ltwo{\delta\rho^{n+1} - \ih{(\delta\rho^{n+1})}} \\
	&\leq C h^{k+1} \left|\delta\rho^{n+1}\right|_{k+1,\,\Omega}\\
	&= C h^{k+1} \left|\int_{t^n}^{t^{n+1}}1\cdot\partial_t\rho\,\mathrm{d}t\right|_{k+1,\,\Omega}\\
	&\leq C h^{k+1} \sqrt{\tau}\Norm{\partial_t\rho}_{L^{2}\left(t^n,t^{n+1};H^{k+1}(\Omega)\right)}\, .
\end{align*}

Collecting all four estimates gives
\begin{align*}
	\Lsnorm{e_h^{n+1}} \leq&\, \Lsnorm{e_h^n} + C\tau^{\frac52}\left(\int_{t^n}^{t^{n+1}}\Ltwo{\partial^3_{t}\rho}^2\mathrm{d}t\right)^{\frac12} +   \frac{C}{2}u_{\infty}  h^k \tau\left|\rho^{n+1}+\rho^n\right|_{k+1, \,\Omega}\\
    &\quad + C h^{k+1} \tau^{\frac12} \left(\int_{t^n}^{t^{n+1}}\left|\partial_t\rho \right|_{k+1,\,\Omega}^2\,\mathrm{d}t\right)^{\frac12},
\end{align*}
which, when added from $n=0$ to $n=N-1$, yields
\begin{align*}
	\Lsnorm{e_h^{N}} \leq&\,   \sum_{n=0}^{N-1}\left\lbrace C \tau^{\frac52}\left(\int_{t^n}^{t^{n+1}}\Ltwo{\partial^3_{t}\rho}^2\mathrm{d}t\right)^{\frac12} +   Cu_{\infty}  h^k \tau\left|\rho^{n+1}+\rho^n\right|_{k+1, \,\Omega}\right.\\
    &\quad + \left. C h^{k+1} \tau^{\frac12} \left(\int_{t^n}^{t^{n+1}}\left|\partial_t\rho \right|_{k+1,\,\Omega}^2\,\mathrm{d}t\right)^{\frac12} \right\rbrace  .
\end{align*}
The second sum can be estimated as
\begin{align*}
\sum_{n=0}^{N-1}Cu_{\infty}  h^k \tau\left|\rho^{n+1}+\rho^n\right|_{k+1, \,\Omega} \leq Cu_{\infty}  h^k \sum_{n=0}^{N} \tau\left|\rho^{n}\right|_{k+1, \,\Omega} ,
\end{align*}
while for the remaining ones we write
\begin{equation*}
	 \tau^{\frac52} \sum_{n=0}^{N-1} \left( \int_{t^n}^{t^{n+1}} \Ltwo{\partial^3_t\rho}^2 \, \mathrm{d}t \right)^{\frac{1}{2}}
	\leq \tau^{\frac52} \left( \sum_{n=0}^{N-1} 1 \right)^{\frac{1}{2}}
	\left( \sum_{n=0}^{N-1} \int_{t^n}^{t^{n+1}} \Ltwo{\partial^3_t\rho}^2 \, \mathrm{d}t \right)^{\frac{1}{2}} 
	= \tau^2  \sqrt{T} \Norm{\partial^3_t\rho}_{L^2(Q)}\,.
\end{equation*}
Similarly,
\begin{align*}
\sum_{n=0}^{N-1}\tau^{\frac12} \left(\int_{t^n}^{t^{n+1}}\left|\partial_t\rho \right|_{k+1,\,\Omega}^2\,\mathrm{d}t\right)^{\frac12}  \leq \sqrt{T} \Norm{\partial_t\rho}_{L^2(0,T;H^{k+1}(\Omega))} .
\end{align*}

Therefore, we get
\begin{align*}
	\Lsnorm{e_h^{N}} &\leq
	\left(C  \sqrt{T} \Norm{\partial_t^3\rho}_{L^2(Q)}\right)\tau^2 + \left(C  \sqrt{T} \Norm{\partial_t\rho}_{L^2(0,T;H^{k+1}(\Omega))}\right)h^{k+1} +  \left(C u_{\infty}  \sum_{n=0}^{N} \tau\left|\rho^{n}\right|_{k+1, \,\Omega}\right)h^k\, .
\end{align*}

It remains to estimate the interpolation error:
\begin{equation*}
	\Lsnorm{\eta^N} = \left(\Norm{\eta^N}_{0,\Omega}^2 + \frac{\tau^2}{4}\Norm{\boldsymbol{u}\cdot\nabla\eta^N}_{0,\Omega}^2\right)^{\frac12}
	\leq C   h^{k+1}\left|\rho^N\right|_{k+1,\,\Omega} + C \tau u_{\infty} h^k\left|\rho^N\right|_{k+1,\,\Omega} \, .
\end{equation*}
Finally, combining the discrete and interpolation errors yields
\begin{align*}
&\Lsnorm{\rho^N - \left(\rho_h^N\right)^+} = \Lsnorm{\eta^N + e_h^N} \leq \Lsnorm{\eta^N}+ \Lsnorm{e_h^N}  \\
&\leq   \left(\sqrt{T} \Norm{\partial_t^3\rho}_{L^2(Q)}\right)C\tau^2 + \left(u_{\infty}\sum_{n=0}^N\tau\left|\rho^n\right|_{k+1,\,\Omega} \right) Ch^k + \left(\sqrt{T}\Norm{\partial_t\rho}_{L^2\left(0,T;H^{k+1}(\Omega)\right)}+\left|\rho^N\right|_{k+1,\,\Omega} \right)Ch^{k+1},
\end{align*}
which concludes the proof.
\end{proof}

\section{The algebraic system: accelerated iterative solution}\label{Sec:Anderson}

To solve the nonlinear problem defined in \eqref{LS-NBPM} at each time step $t^n$, we start by recalling the definition \eqref{LS-NBPM-bis} of the bilinear form $a^{n+1}(\cdot,\cdot)$ and the linear functional $L^{n+1}(\cdot)$. Then, we employ the following fixed-point Richardson iteration with damping parameter $\omega \in (0,1]$: starting from the initial guess $\tilde{\rho}^{(0)}:= \rho_h^{n}$, find $\tilde{\rho}^{(\texttt{k}+1)} \in V_h$ such that

\begin{equation}\label{Eq:Richardson_it_SH}
a^{n+1}(\tilde\rho^{(\texttt{k}+1)}, \phi_h) \;=\; a^{n+1}(\tilde\rho^{(\texttt{k})},\phi_h) \;+\; \omega\Big[\, L^{n+1}(\phi_h) \;-\; a^{n+1}\big((\tilde\rho^{(\texttt{k})})^+ ,\phi_h \big) \;-\; s\big((\tilde\rho^{(\texttt{k})})^-,\, \phi_h\big) \,\Big],
\end{equation}
for all $\phi_h\in V_h$, and 
for $\texttt{k} = 0, 1, \ldots, N_{\max}$, or 
until the stopping criterion
\begin{equation}\label{eq:rich_tolerance}
\| \tilde{\rho}^{(\texttt{k}+1)} - \tilde{\rho}^{(\texttt{k})} \|_{0,\Omega} \leq \texttt{tol}\,,
\end{equation}
is satisfied. Then, we set $\rho_h^{n+1} = \tilde{\rho}^{(\texttt{k}+1)}$. Importantly, the symmetry of the bilinear form $a^{n+1}(\cdot,\cdot)$ makes the linear system \eqref{Eq:Richardson_it_SH} symmetric. This can be exploited by using fast solvers such as Cholesky decomposition (for direct solvers) or conjugate gradient (for iterative ones). 

Beyond only exploiting its symmetry, further speed-up can be gained through acceleration methods, which are well known for substantially reducing the number of fixed-point iterations needed to reach a solution that fulfills the imposed tolerance. In this paper, we use the \textit{Anderson Acceleration (AA)} method, first introduced in \cite{Anderson-1965}, and then reformulated in \cite{Walker-Ni-2011} and \cite{Saad_2025}. In this work, we follow the reformulation proposed in \cite{Saad_2025}. The idea is to replace the plain Richardson iterate with a linear combination of the $m$ most recent iterates, chosen to minimise the norm of the corresponding
linear combination of residuals. 

Let $\boldsymbol\rho^{(\texttt{k})}$
denote the vector representation of $\tilde\rho^{(\texttt{k})}$, and let $G(\boldsymbol\rho^{(\texttt{k})})$
denote the corresponding representation of the solution $\tilde\rho^{(\texttt{k}+1)}$ of
\eqref{Eq:Richardson_it_SH}. Additionally, define
$\mathbf f^{(\texttt{k})} = G(\boldsymbol\rho^{(\texttt{k})}) - \boldsymbol\rho^{(\texttt{k})}$ as the corresponding increment between the current plain Richardson iteration solution and the previous Anderson Accelerated solution. Given a window of depth $m$, only
the most recent iterates are retained, so that at iteration $\texttt{k}$ the
effective depth is $m_{\texttt{k}} = \min(m, \texttt{k})$. Then, we define the increment matrices
\begin{equation}
  X_{\texttt{k}} = \big[\Delta\boldsymbol\rho^{(\texttt{k}-m)},\dots,\Delta\boldsymbol\rho^{(\texttt{k}-1)}\big],
  \qquad
  F_{\texttt{k}} = \big[\Delta\mathbf f^{(\texttt{k}-m)},\dots,\Delta\mathbf f^{(\texttt{k}-1)}\big],
\end{equation}
with $\Delta(\cdot)^{(i)} := (\cdot)^{(i+1)}-(\cdot)^{(i)}$. The extrapolation coefficients
$\boldsymbol\gamma^{(\texttt{k})}$, also called mixing coefficients, are obtained by solving the least-squares problem
\begin{equation}\label{Eq:AA_lsq}
  \boldsymbol\gamma^{(\texttt{k})} = \operatorname*{arg\,min}_{\boldsymbol\gamma\in \mathbb{R}^{m_{\text{k}}}}
  \big\|\mathbf f^{(\texttt{k})} - F_{\texttt{k}}\boldsymbol\gamma\big\|_2,
\end{equation}
which we can solve through the normal equations
\begin{equation}\label{Eq:AA_normal_eq}
  \big(F_{\texttt{k}}^\top F_{\texttt{k}} \big)\,\boldsymbol\gamma^{(\texttt{k})}
  = F_{\texttt{k}}^\top \mathbf f^{(\texttt{k})}.
\end{equation}

Since $F_{\texttt{k}}^\top F_{\texttt{k}}$ can become ill-conditioned near convergence, following \cite{Saad_2025} we regularise \eqref{Eq:AA_normal_eq} and solve
\begin{equation}\label{Eq:AA_lsq_reg}
  \big(F_{\texttt{k}}^\top F_{\texttt{k}} + \mathcal{E}\|F_{\texttt{k}}^\top F_{\texttt{k}}\|_F\, I\big)\,\boldsymbol\gamma^{(\texttt{k})}
  = F_{\texttt{k}}^\top \mathbf f^{(\texttt{k})},
\end{equation}
with $\mathcal{E} = 10^{-10}$ throughout this work. 

Finally, the accelerated iterate is given by
\begin{equation}\label{Eq:AA_update}
  \boldsymbol\rho^{(\texttt{k}+1)} = G(\boldsymbol\rho^{(\texttt{k})}) - (X_{\texttt{k}}+F_{\texttt{k}})\,\boldsymbol\gamma^{(\texttt{k})}.
\end{equation}
Setting $m=0$ recovers the plain Richardson iteration
\eqref{Eq:Richardson_it_SH}. The full Anderson-accelerated Richardson iteration is described in Algorithm \ref{alg:anderson-richardson}. 

\begin{algorithm}[h!]
\caption{Anderson-accelerated Richardson iteration}
\label{alg:anderson-richardson}
\begin{algorithmic}[1]
\REQUIRE initial guess $\boldsymbol\rho^{(0)} = \boldsymbol\rho^n$, window depth $m$,
tolerance $\mathrm{tol}$, maximum iterations $N_{\max}$, regularisation
parameter $\mathcal{E}$
\ENSURE accelerated solution
\STATE Initialise history buffers $\mathcal X \gets \emptyset$, $\mathcal F \gets \emptyset$
\FOR{$\texttt{k} = 0, 1, \dots, N_{\max}$}
    \STATE Solve \eqref{Eq:Richardson_it_SH} for $G(\boldsymbol\rho^{(\texttt{k})})$
    \STATE $\mathbf f^{(\texttt{k})} \gets G(\boldsymbol\rho^{(\texttt{k})}) - \boldsymbol\rho^{(\texttt{k})}$
    \IF{$\| \tilde{\rho}^{(\texttt{k}+1)} - \tilde{\rho}^{(\texttt{k})} \|_{0,\Omega} < \texttt{tol}$}
        \RETURN $G(\boldsymbol\rho^{(\texttt{k})})$
        \COMMENT{converged}
    \ENDIF
    \IF{$\texttt{k} = 0$}
        \STATE $\boldsymbol\rho^{(\texttt{k}+1)} \gets G(\boldsymbol\rho^{(\texttt{k})})$
        \COMMENT{plain Richardson step, no acceleration}
    \ELSE
        \STATE Append $\boldsymbol\rho^{(\texttt{k})}, \mathbf f^{(\texttt{k})}$ to $\mathcal X, \mathcal F$;
               discard oldest entries if $|\mathcal F| > m$
        \STATE $m_{\texttt{k}} \gets |\mathcal F|$
        \IF{$m_{\texttt{k}} > 1$}
            \STATE Form $X_{\texttt{k}}, F_{\texttt{k}}$ from consecutive differences of $\mathcal X, \mathcal F$
            \STATE Solve $\big(F_{\texttt{k}}^\top F_{\texttt{k}} + \mathcal{E}\|F_{\texttt{k}}^\top F_{\texttt{k}}\|_F\, I\big)\,\boldsymbol\gamma^{(\texttt{k})} = F_{\texttt{k}}^\top \mathbf f^{(\texttt{k})}$
            \STATE $\boldsymbol\rho^{(\texttt{k}+1)} \gets G(\boldsymbol\rho^{(\texttt{k})}) - (X_{\texttt{k}}+F_{\texttt{k}})\,\boldsymbol\gamma^{(\texttt{k})}$
        \ELSE
            \STATE $\boldsymbol\rho^{(\texttt{k}+1)} \gets G(\boldsymbol\rho^{(\texttt{k})})$
            \COMMENT{insufficient history for acceleration}
        \ENDIF
    \ENDIF
\ENDFOR
\RETURN $\boldsymbol\rho^{(N_{\max}+1)}$
\COMMENT{$N_{\max}$ reached without convergence}
\end{algorithmic}
\end{algorithm}

\section{Numerical experiments}\label{sec:Numer}

This section presents several numerical examples to verify our theory and also compare standard least-squares methods to the corresponding nodally bound-preserving formulation. We assess convergence, stability, mass conservation, and other qualitative aspects of the approximations, such as smoothness and sharpness. All tests considered simplicial Lagrangian finite elements, with the convective velocity either imposed analytically or given as part of an incompressible flow system. 
To keep the discussion and comparative studies as clear as possible, we use throughout the numerical experiments a tolerance of $\texttt{tol}=10^{-8}$ in \eqref{eq:rich_tolerance}. In addition, all the linear  systems are solved directly through a sparse Cholesky factorisation provided by the Multifrontal Massively Parallel sparse direct Solver (MUMPS); however, applying a preconditioned conjugate gradient method to solve \eqref{Eq:Richardson_it_SH} would be straightforward and generally more efficient. The numerical experiments were implemented in FEniCSx (DOLFINx 0.10.0) \cite{BarattaEtal-2023, ScroggsEtal-2022, BasixJoss-2022, AlnaesEtal-2014}, an open-source finite element library for the solution of PDEs. MUMPS was accessed through the PETSc library \cite{petsc-user-ref-2026} via its Python interface \texttt{petsc4py} \cite{dalcinpazklercosimo-2011}. 
The implementation is publicly available on GitLab.\footnote{\codeurl}

\subsection{Rotation of a smoothed step}
The first set of tests considers a smoothed step in pure rotation, where $\Omega=\{(x,y)\in\mathbb{R}^2: x^2 + y^2 < (\nicefrac12)^2\}$. The velocity field $\bbeta = 2\pi(-y,x)^{\top}$ is constant in time, and the initial solution is
\begin{align*}
\rho_0(x, y) = \frac{1}{2} + \frac{1}{2} \tanh\!\left(\varkappa_{\rm sharp}\left(0.1 - \sqrt{(x + 1/4)^2 + y^2}\right) \right),
\end{align*}
where $\varkappa_{\rm sharp}>0$ is a smoothness parameter (the limit $\varkappa_{\rm sharp}\rightarrow \infty$ recovers a discontinuous circular step). This corresponds approximately to a circular ``unit bump'' with radius $0.1$, centred at $(x,y) = (-0.25,0)$. The exact solution is thus a simple counter-clockwise rotation of $\rho_0$ according to an angle $\Theta(t) = 2\pi t$.

\subsubsection{Convergence}
We first take $\varkappa_{\rm sharp}=20$ and study convergence in space and time separately, considering a refinement index $j = 1,\dots,r$, where $r$ denotes the number of refinement levels. For the spatial study, we fix $\tau=10^{-4}$ and consider structured meshes with mesh sizes $h_j = 10^{-1}\times 2^{-(j-1)}$, with $r_{h}=6$ levels of refinement. For the temporal study, we fix $h=2.5\times10^{-3}$ and consider time-step sizes $\tau_j = 10^{-2}\times 2^{-(j-1)}$, with $r_{\tau}=7$ levels of refinement. We use a relaxation parameter $\omega = 0.05$ for the Richardson iteration, a history depth $m = 10$ for Anderson acceleration and $\alpha=1.0$ in the stabilisation form.
Figure \ref{fig:case1_L2_convergence} describes the asymptotic behaviour of the $L^2(\Omega)$ error using $\mathbb{P}_1$ and $\mathbb{P}_2$ elements. According to Theorem \ref{Theo:Convergence}, fixing a sufficiently small $\tau$ and refining the mesh size should yield at least an $O(h^k)$ convergence using $\mathbb{P}_k$ elements, which is confirmed by the results depicted in Figure~\ref{fig:case1_L2_vs_h} (although with a higher power of $h$ than predicted by the theory, which might hint at a superconvergence for this problem).
Then, with the mesh fixed and varying $\tau$, second-order temporal convergence was observed for both $\mathbb{P}_1$ and $\mathbb{P}_2$ elements, as shown in Figure \ref{fig:case1_L2_vs_tau}. 

\begin{figure}[htbp]
    \centering
    \begin{subfigure}[b]{0.48\textwidth}
        \centering
        \includegraphics[width=\textwidth]{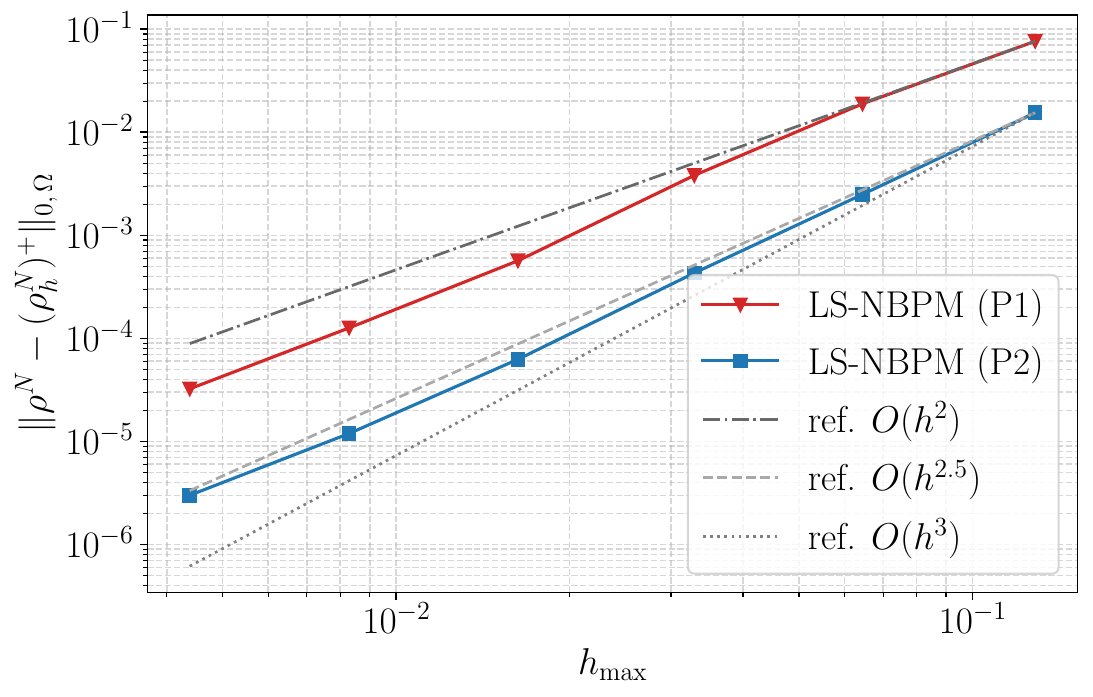}
        \caption{$L^2(\Omega)$ error at $t=1$ over mesh refinement.}
        \label{fig:case1_L2_vs_h}
    \end{subfigure}
    \hfill
    \begin{subfigure}[b]{0.48\textwidth}
        \centering
        \includegraphics[width=\textwidth]{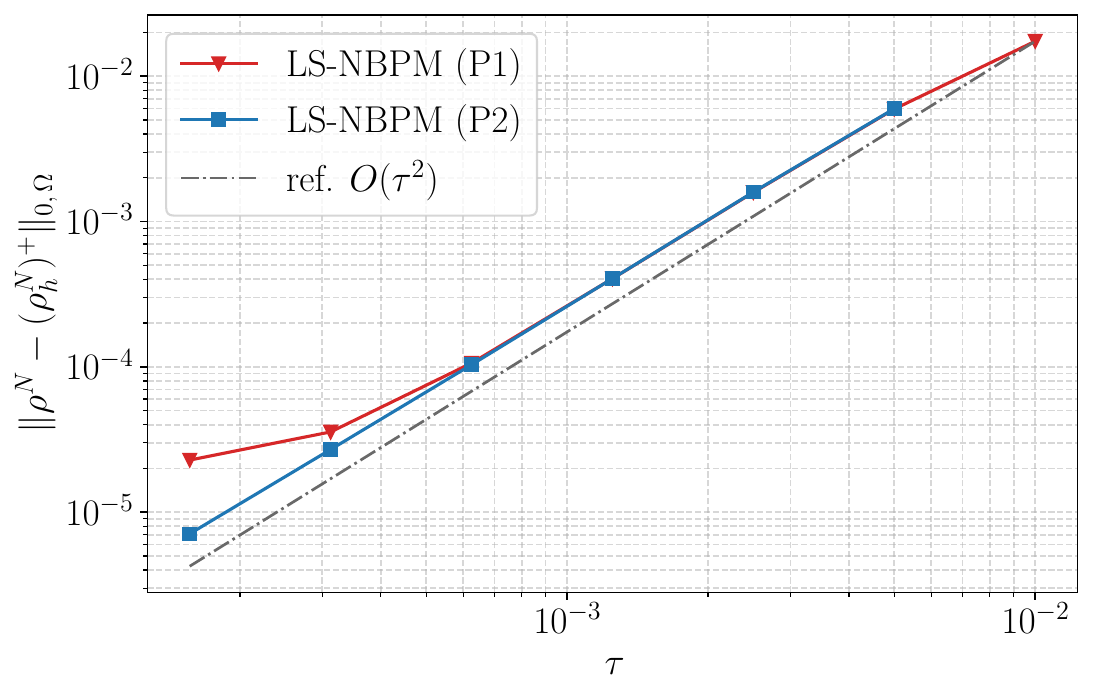}
         \caption{$L^2(\Omega)$ error at $t=1$ over temporal refinement.}
        \label{fig:case1_L2_vs_tau}
    \end{subfigure}
    \caption{Convergence of the error $\| \rho^N -  (\rho^N_h)^+\|_{0,\Omega}$ under mesh refinement (left) and time-step size refinement (right) for $\mathbb{P}_1$ and $\mathbb{P}_2$ elements.}
    \label{fig:case1_L2_convergence}
\end{figure}

\subsubsection{Conservation: mass and bounds}

Since $\dive\bu=0$ for all $t$, the exact solution of this problem conserves both the mass and the initial extrema, i.e., $\rho$ satisfies
\begin{align*}
\underbrace{\int_{\Omega}\rho(\bx,t)\,\mathrm{d}\Omega}_{\mathrm{M}(t)} = \underbrace{\int_{\Omega}\rho_0\,\mathrm{d}\Omega}_{\mathrm{M}(0)} 
\quad \text{and} \quad
	\rho(\bx,t) \in \big[\rho_0^{\mathrm{min}},\rho_0^{\mathrm{max}}\big] \approx [0,0.982]     \ \ \forall\, t\geq 0 \, .
\end{align*}
Our method fulfills the latter, but not (exactly) the former. Conversely,  the baseline discretisation \eqref{LS-FEM} conserves mass, but not the bounds. Figure \ref{fig:case1_masschange} shows the relative mass change, defined as $\Delta M\,[\%] = (\mathrm{M}(t)/\mathrm{M}(0) - 1) \times  100\% $, for both cases; mesh refinement with fixed $\tau$, and temporal refinement with fixed mesh. The $\tau$ and $h$ are considered the same as in the previous paragraph. Both cases show good mass conservation across the refinement levels, with a maximum of around $6.4\,\%$ of added mass for the  coarsest mesh. 

\begin{figure}[htbp]
    \centering
    \begin{subfigure}[b]{0.48\textwidth}
        \centering
        \includegraphics[width=\textwidth]{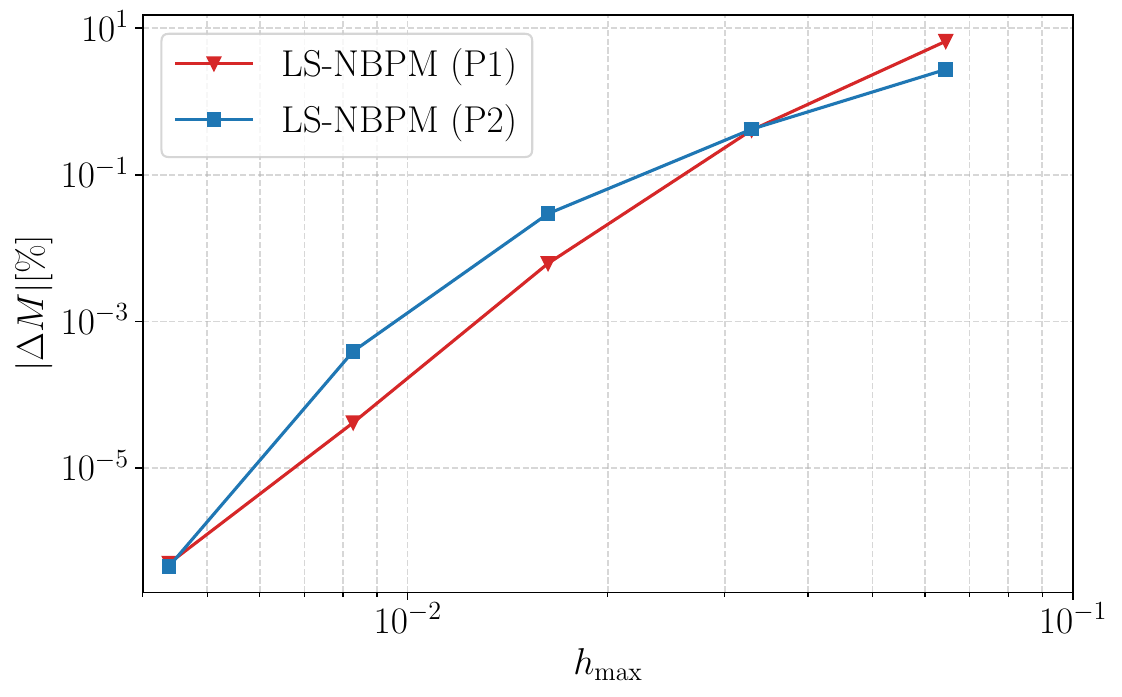}
        \caption{$\Delta M [\%]$ over mesh refinement.}
        \label{fig:case1_masschange_vs_h}
    \end{subfigure}
    \hfill
    \begin{subfigure}[b]{0.48\textwidth}
        \centering
        \includegraphics[width=\textwidth]{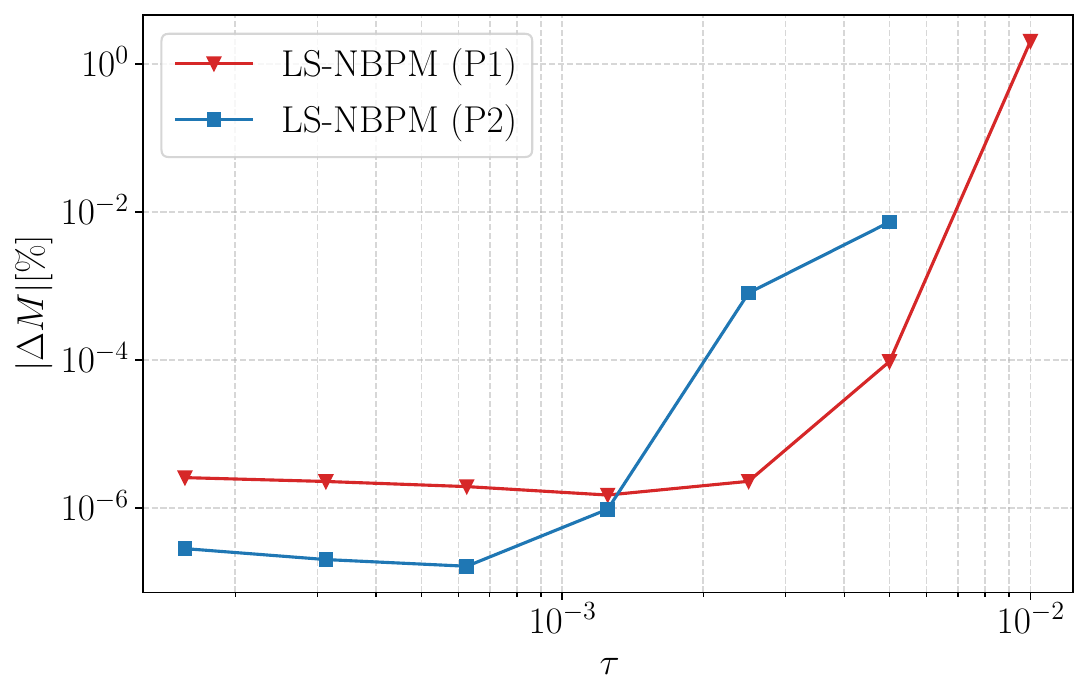}
         \caption{$\Delta M [\%]$ over time-step size refinement.}
        \label{fig:case1_masschange_vs_tau}
    \end{subfigure}
    \caption{Relative mass change over mesh refinement (left) and temporal refinement (right) for $\mathbb{P}_1$ and $\mathbb{P}_2$ elements.}
    \label{fig:case1_masschange}
\end{figure}

 \subsubsection{Accelerated fixed-point iteration}

The benefits of the accelerated fixed-point iteration given by the AA method are showcased in Figure~\ref{fig:case1_richardson_iters_vs_time}. It shows the number of Richardson iterations required to reach the tolerance defined in Eq.~\eqref{eq:rich_tolerance}, for a fixed $\tau = 10^{-4}$ and mesh refinement levels 2, 3, and 4, using  $\mathbb{P}_1$ elements. Here, a window of $m=10$ previous iterations was considered for the AA method, and $\omega=0.1$ as damping factor. The number of iterations using the AA method was considerably reduced---by around an order of magnitude in each case. More detailed results are shown in Table~\ref{tab:elapsed_time_richardson_iters}, which presents the average number of fixed-point iterations needed to converge for the cases $m=0$ and $m=10$, for a sequence of meshes. The substantial reduction in the iteration count highlights the potential for improvement in the fixed-point iteration method applied, and the applicability of the current LS-NBP method to more complex cases.

\begin{table}[htbp]
  \centering
  \begin{tabular}{c|cc|c|c}
    \toprule
    Mesh factor & DOFs & $h_{\max}$ & $m$ & Avg. Richardson Iters. \\
    \midrule
    \multirow{2}{*}{1} & \multirow{2}{*}{114} & \multirow{2}{*}{$1.28\times10^{-1}$} & 0 & 94 \\
     &  &  & 10 & 12 \\
    \midrule
    \multirow{2}{*}{2} & \multirow{2}{*}{414} & \multirow{2}{*}{$6.44\times10^{-2}$} & 0 & 104 \\
     &  &  & 10 & 11 \\
    \midrule
    \multirow{2}{*}{3} & \multirow{2}{*}{1,540} & \multirow{2}{*}{$3.29\times10^{-2}$} & 0 & 113 \\
     &  &  & 10 & 11 \\
    \midrule
    \multirow{2}{*}{4} & \multirow{2}{*}{5,976} & \multirow{2}{*}{$1.63\times10^{-2}$} & 0 & 122 \\
     &  &  & 10 & 10 \\
    \midrule
    \multirow{2}{*}{5} & \multirow{2}{*}{23,600} & \multirow{2}{*}{$8.29\times10^{-3}$} & 0 & 131 \\
     &  &  & 10 & 9 \\
    \bottomrule
  \end{tabular}
  \caption{Average Richardson iterations for $m=0$ (plain) vs.~$m=10$ (accelerated).}
  \label{tab:elapsed_time_richardson_iters}
\end{table}

\begin{figure}[htbp]
  \centering
\includegraphics[width=0.9\textwidth]{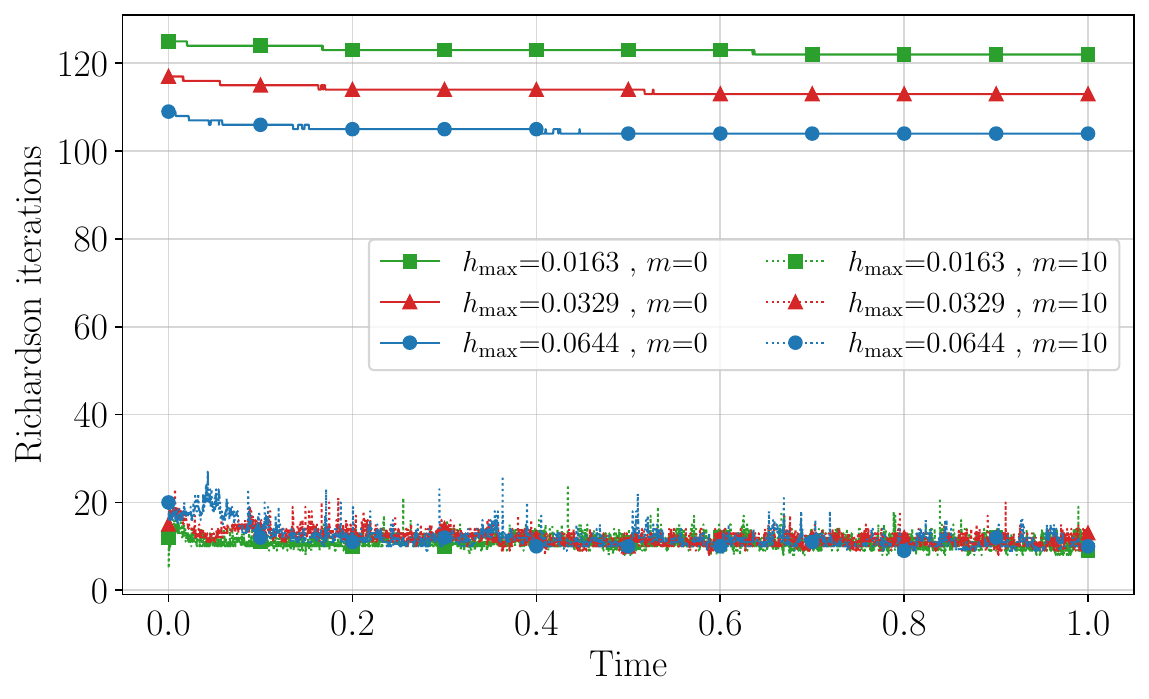}
  \caption{Richardson iterations over time for $\tau = 10^{-4}$, using meshes $M_i \, ; \; i \in \lbrace 2,3,4\rbrace$, for plain Richardson iterations ($m=0$), and Anderson-Accelerated Richardson with $m=10$.}
  \label{fig:case1_richardson_iters_vs_time}
\end{figure}

\subsubsection{Stability for a near-discontinuous case}

The following results consider $\varkappa_{\rm sharp} = 100$, which yields a very sharp, nearly discontinuous bump. We use $\tau = 10^{-3}$ and $\mathbb{P}_1$ elements. Two meshes were analysed: those corresponding to refinement levels 4 ($M_4$) and 5 ($M_5$) of the mesh refinement study, with approximately $5,976$ and $23,600$ degrees of freedom (DOFs), respectively.
Figures ~\ref{fig:case1_snapshots_2d} and \ref{fig:case1_snapshots_3d} show snapshots of the numerical solutions at time $t = 1$, after one complete ``lap'' for mesh $M_4$. Here, the green line in Figure \ref{fig:case1_snapshots_2d} shows the isocontour where $\rho_h=0.5$.  
To compare the methods quantitatively, Figure~\ref{fig:case1_densityovercircularline} shows the numerical solutions along the circle of radius $r=0.25$ centred at $(0,0)$, parametrised by the angle $\beta \in [0,2\pi)$. Despite providing stabilisation with respect to the plain Galerkin method, we observe that the baseline LS solution severely violates the bounds and presents oscillations near the sharp fronts. On  the other hand, the current method delivers a discrete solution that remains within the admissible range $(0,1)$, where only a very minor oscillation, localised in the wake of the step, can be observed.

\begin{figure}[htbp]
  \centering
\includegraphics[width=0.9\textwidth]{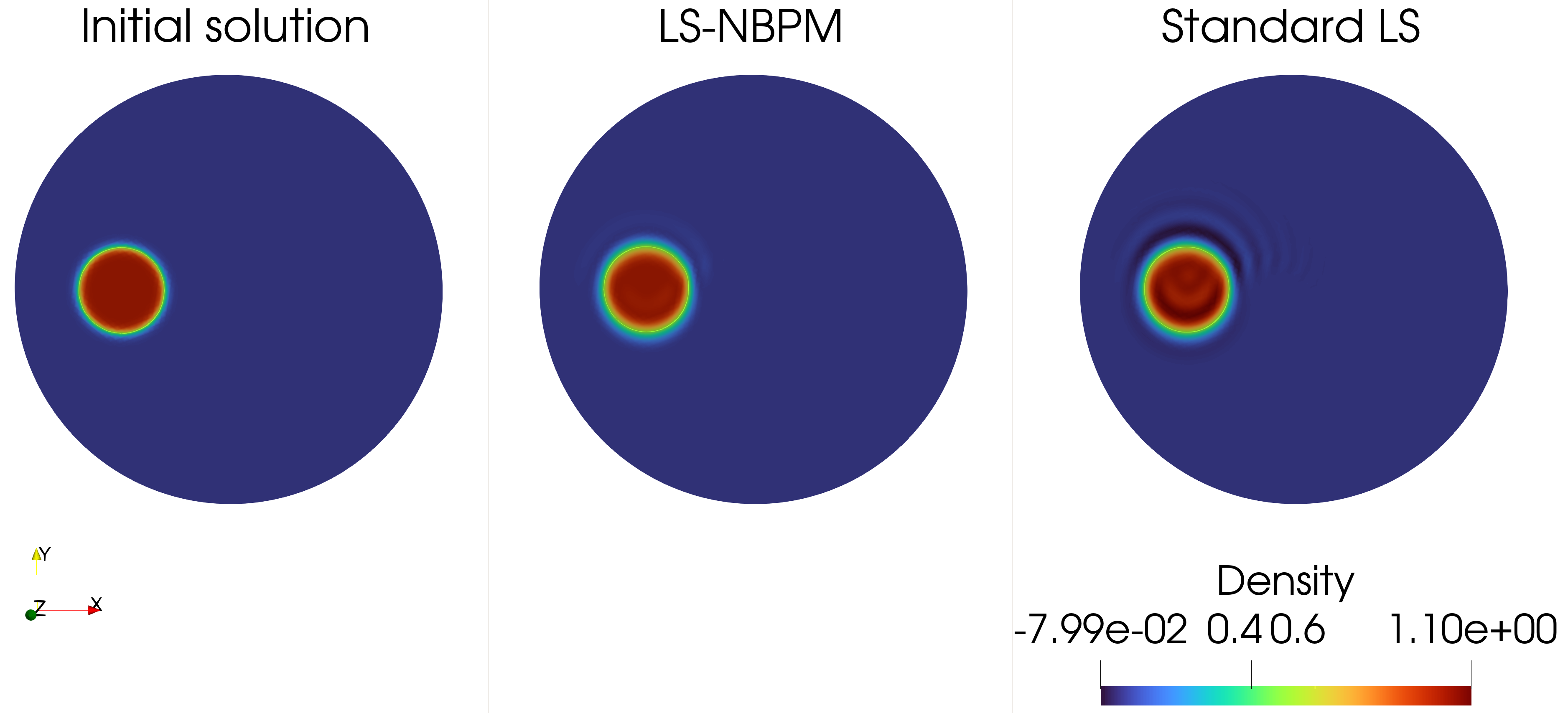}
  \caption{Snapshot of the numerical solution at $t = 1$ using mesh $M_4$, 2D top view.}
  \label{fig:case1_snapshots_2d}
\end{figure}

\begin{figure}[htbp]
  \centering
\includegraphics[width=0.9\textwidth]{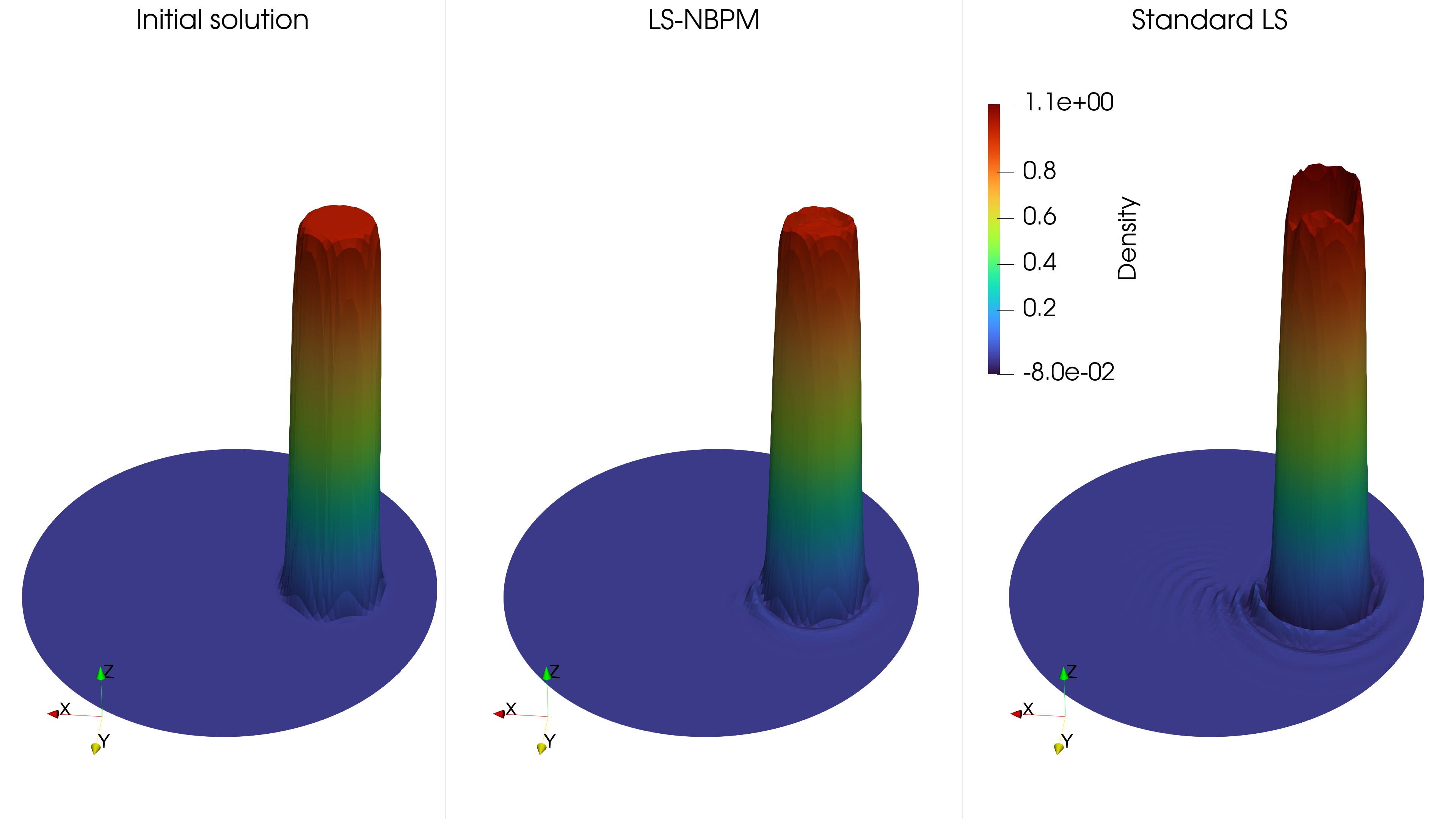}
  \caption{Snapshot of the numerical solution after one lap ($t = 1$) for mesh $M_4$. The solution is shown as a 3D surface, with height proportional to $\rho$. The reference solution on the left is the $M_4$-interpolation of the exact solution.}
  \label{fig:case1_snapshots_3d}
\end{figure}

\begin{figure}[htbp]
    \centering
    \begin{subfigure}[b]{0.48\textwidth}
        \centering
        \includegraphics[width=\textwidth]{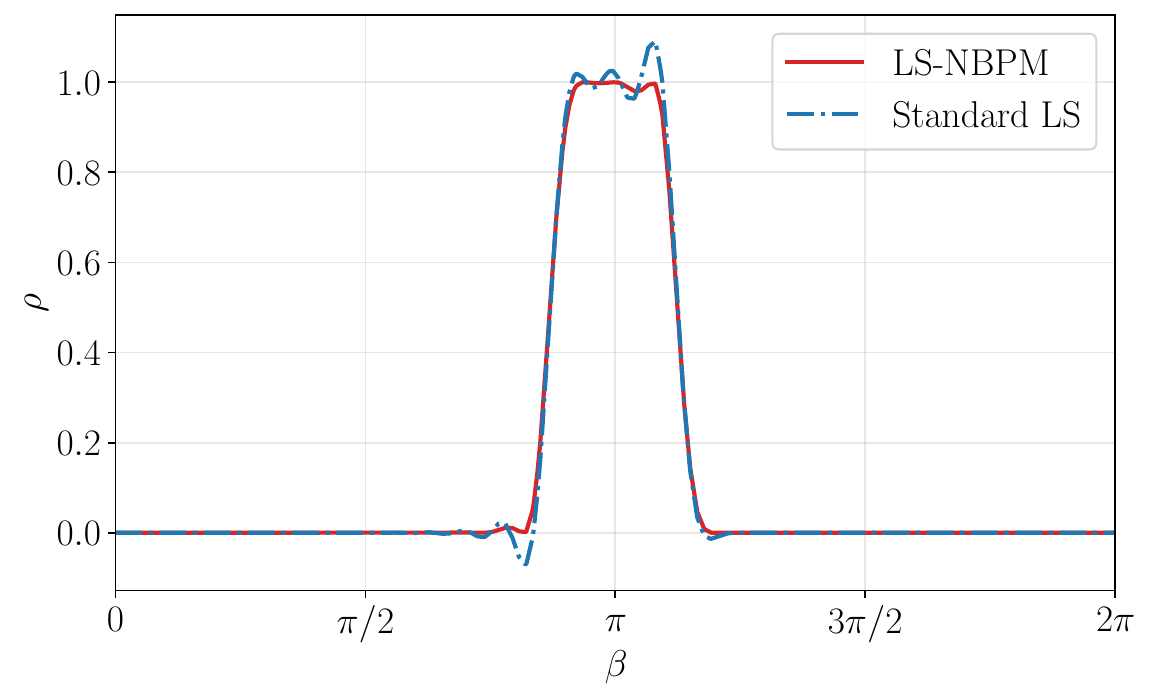}
        \caption{Discrete solutions over circular line for mesh $M_4$.}
        \label{fig:case1_densityovercircularline_factor3}
    \end{subfigure}
    \hfill
    \begin{subfigure}[b]{0.48\textwidth}
        \centering
        \includegraphics[width=\textwidth]{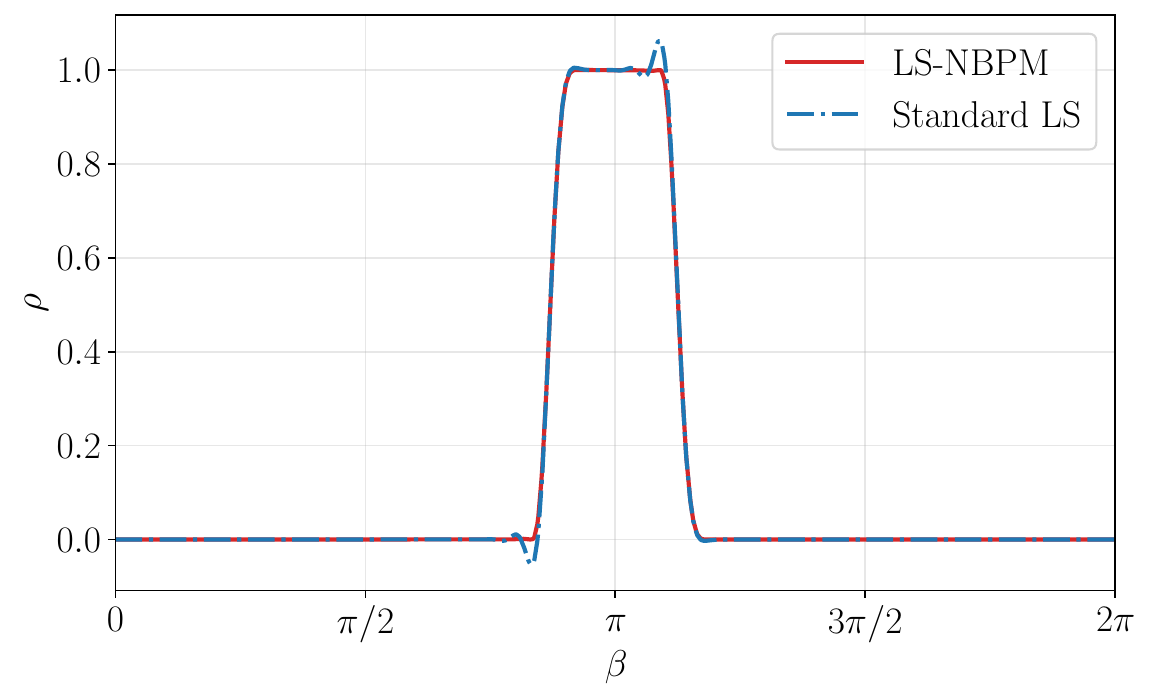}
         \caption{Discrete solutions over circular line for mesh $M_5$.}
         \label{fig:case1_densityovercircularline_factor4}
    \end{subfigure}
    \caption{Numerical solutions at time $t=1$ along a circular line centred at $(0,0)$, with radius $r=0.25$, parametrised by an angle $\beta\in [0,2\pi)$. The results consider a middle-coarse mesh (left), and a middle-fine mesh (right).}
    \label{fig:case1_densityovercircularline}
\end{figure}

\subsection{A two-phase flow benchmark}
\label{sec:2phase}

We finally widen the applicability of our method by applying it to the Navier--Stokes system with variable density.  The so-called Rayleigh--Taylor instability is a popular benchmark for variable-density flow solvers \cite{Guermond2000}. The setup has two fluids initially at rest in $\Omega = (0,a/2)\times(-2a,2a)\subset\R^2, a>0$,  under gravity.  We decompose the boundary as $\partial\Omega=\Gamma_\ell\cup\Gamma_{\rm tb}$,
where $\Gamma_{\rm tb}:= \{(x,y)\in\partial\Omega: |y|=2a\}$ and $\Gamma_{\ell}:= \{(x,y)\in \partial\Omega: x=0 \;\textrm{or}\; x=a/2\}$. The system of equations is a consistent reformulation
derived in \cite{Guermond2000}: Find the velocity $\bu$, pressure $p$, and
density $\rho$ such that
\begin{align}
\partial_t\rho + \bu\cdot\nabla\rho + \frac{\dive\bu}{2}\rho &= 0 && \textrm{in}\;(0,T]\times\Omega\,,\label{RT-1}\\
\sqrt{\rho}\,\partial_t\big(\sqrt{\rho}\bu\big) + (\rho\nabla\bu)\bu + \frac{1}{2}[\dive(\rho\bu)]\bu -
\dive(2\nu\rho\boldsymbol{\varepsilon}(\bu)) + \nabla p &= \rho\boldsymbol{g} && \textrm{in}\;(0,T]\times\Omega\,,\label{RT-2}\\
\dive\bu &= 0 &&\textrm{in}\;(0,T]\times\Omega\,,\label{RT-3}\\
\bu &= \boldsymbol{0} &&\textrm{in}\;(0,T]\times\Gamma_{\rm tb}\,,\label{RT-4}\\
(\boldsymbol{\varepsilon}(\bu)\bn)\cdot\boldsymbol{t} = 0 \quad \textrm{and}\quad\bu\cdot\bn &= 0 &&\textrm{in}\;(0,T]\times \Gamma_{\ell}\,,\label{RT-5}\\
\rho = \rho_0^{}\quad\textrm{and}\quad \bu &= \boldsymbol{0} &&\textrm{in}\;\{0\}\times\Omega\,.\label{RT-6}
\end{align}
Here, $\nu>0$ is the constant kinematic viscosity of the fluid,  $\boldsymbol{\varepsilon}(\bu)
=(\nabla\bu + \nabla\bu^{\top})/2$ is the symmetric velocity gradient, and $\boldsymbol{g}=(0,-g)^{\top}$ denotes the downward-pointing constant gravity. The skew-symmetrising term $(\nicefrac12\dive\bu)\rho$ added to \eqref{RT-1} guarantees unconditional stability for $\rho$ even if the velocity $\boldsymbol{u}$ is not \textsl{exactly} solenoidal \cite{Guermond2000}.

In this setup, the denser fluid, II, sits initially on top of the lighter fluid, I ($\rho_{\mathrm{II}} > \rho_{\mathrm{I}}$), according to the initial density field
\begin{align*}
\rho_0(x,y) = \frac{\rho_{\mathrm{II}}+\rho_{\mathrm{I}}}{2} + \frac{\rho_{\mathrm{II}} - \rho_{\mathrm{I}}}{2}\tanh\left(100\frac{y}{a} + 10\cos \frac{2\pi x}{a}\right) .
\end{align*}
The flow regime is parametrised by the Reynolds and Atwood numbers
\begin{align*}
\text{Re} = \frac{a\sqrt{ag}}{\nu} = 1000 \ \ \text{and} \ \ \text{At} = \left|\frac{\rho_{\mathrm{II}} - \rho_{\mathrm{I}}}{\rho_{\mathrm{II}} + \rho_{\mathrm{I}}}\right| = \frac{1}{2}\, ,
\end{align*}
so that $\rho_{\mathrm{II}} = 3\rho_{\mathrm{I}}$.  %

We use a second-order IMEX scheme that, at each time step, first solves for density using extrapolated velocities, then updates velocity and pressure from a linearised discretisation of the momentum equation. Defining $\hat{\bu}^{n+1}_h := 2\bu^n_h - \bu^{n-1}_h$, we use an IMEX Crank--Nicolson method to evolve the density, then a second-order IMEX backward scheme for the momentum equation: Given $\rho^0_h=i_h(\rho_0)$ and
$\bu_h^0=\boldsymbol{0}$, find $(\rho_h^{n+1},\bu_h^{n+1},p_h^{n+1})$, for $n=2,3,\dots,N-1$, by solving
\begin{align}
& \inner {\frac{(\rho_h^{n+1})^+- (\rho_h^{n})^+}{\tau} + \frac{1}{2}\mathcal{C}^{n+1}(\rho_h^{n+1})^+ + \frac{1}{2}\mathcal{C}^{n}(\rho_h^{n})^+}{ \frac{\phi_h}{\tau}+\frac{1}{2}\mathcal{C}^{n+1}\phi_h}
+ \tau^{-2}s((\rho_h^{n+1})^-,\phi_h)=0\,,\\
&\frac{1}{2\tau}\!\left\langle\sqrt{(\rho_h^{n+1})^+}\left[3\sqrt{(\rho_h^{n+1})^+}\,\bu_h^{n+1}-4\sqrt{(\rho_h^n)^+}\,\bu_h^{n} + \sqrt{(\rho_h^{n-1})^+}\,\bu_h^{n-1}\right],\bv_h\right\rangle_\Omega  \nonumber\\
&+ \left\langle \big[(\rho_h^{n+1})^+\nabla\bu_h^{n+1}\big]\hat{\bu}_h^{n+1} + \frac{\dive\!\big[(\rho_h^{n+1})^+\hat{\bu}_h^{n+1}\big]}{2}\bu_h^{n+1},\bv_h\right\rangle_\Omega  \nonumber\\
& + 2\nu\!\left\langle(\rho_h^{n+1})^+\boldsymbol{\varepsilon}(\bu_h^{n+1}),\boldsymbol{\varepsilon}(\bv_h)\right\rangle_\Omega
 -\langle p_h^{n+1},\dive \bv_h\rangle_\Omega + \langle q_h,\dive \bu_h^{n+1}\rangle_\Omega=   \langle(\rho_h^{n+1})^+(0,-g)^{\top},\bv_h\rangle_\Omega\, ,\label{divuBDF2}
\end{align}
where $\mathcal{C}^{n}\phi := \hat{\bu}_h^{n}\cdot\nabla\phi +(\dive\hat{\bu}_h^{n}/2)\,\phi$. Since we do not have $\bu_h^{-2} $ and $\bu_h^{-1}$, the method is not self-starting; the first two steps are therefore computed using backward Euler instead. The unconditional stability for this scheme can be proved using, e.g., the techniques presented in \cite{Espinoza2025}.

Here, using a bound-preserving scheme is indispensable: if the density becomes negative, the square roots in Eq.~\eqref{divuBDF2} cannot be computed, so the method breaks down. To test that, we run the simulation with and without bound preservation, for different meshes. The admissible interval enforced for the density is $[\rho_{\mathrm{I}},\rho_{\mathrm{II}}] = [1,3]$, as fulfilled by the exact solution $\rho$ for all $t$. We use the inf-sup-stable Taylor--Hood element $\mathbb{P}_2^2\times \mathbb{P}_1^{}$ for velocity and pressure, and $\mathbb{P}_1$ for the density. We have used three meshes, $M_{RT1}$, $M_{RT2}$ and $M_{RT3}$, with $25\times200$, $50\times400$ and $100\times800$ elements, respectively, and a time-step size $\tau\sqrt{g/a} = 10^{-3}$. Furthermore, we set $\alpha = 1.0$ in the stabilisation form, $\omega = 0.05$ in the Richardson iteration, and a history depth $m = 10$ for the Anderson acceleration.

In Figure~\ref{fig:case2_max_min_density_over_time} we depict the evolution of the minimum and maximum values for the discrete density for both the baseline LS discretisation and the current method. While we see that the current method delivers a solution that respects the bounds throughout all the computational time, the baseline violates them almost from the start of the computation. Moreover, for all meshes there was a time at which the minimum density provided by the LS method becomes negative. At that point, the momentum solver breaks down due to coefficients $\sqrt{\rho}$, which can no longer be computed. The time at which the first negative value is reached increases modestly with mesh refinement ($t\sqrt{g/a}=2.4, 3$, and $3.2$, respectively), while the number of density degrees of freedom quadruples with each refinement.

Figure~\ref{fig:bubbleHeight} depicts how the top and bottom bubbles move in time, comparing our method, the baseline least-squares discretisation, and reference values from \cite{Guermond2000} and \cite{Try88}. All the methods seem to give similar results in this metric (at least for as long as the reference data go). Nevertheless, a very interesting picture appears when we show cross sections of the density in Figures \ref{fig:case2_MRT1_densityoververtialline_t2and265} and \ref{fig:case2_densityoververtialline_t2and3}. The solution given by the baseline LS discretisation is polluted with spurious oscillations, which were virtually eliminated by our method. Figure~\ref{fig:snapshots} shows the discrete solution for both methods; we observe the intricate pattern of density that explains the local variations (within the admissible range) between the two methods.

These results agree with previous numerical evidence on NBP methods, hinting at the fact that, in addition to ensuring that the discrete solution remains within the admissible values, the nodally bound-preserving method also seems to have a substantial stabilising effect that tends to remove oscillations altogether. While there is not a formal proof for this fact, there is enough numerical evidence showing that, at the very least, spurious oscillations are all but removed.

\begin{figure}[htbp]
  \centering
\includegraphics[width=0.8\textwidth]{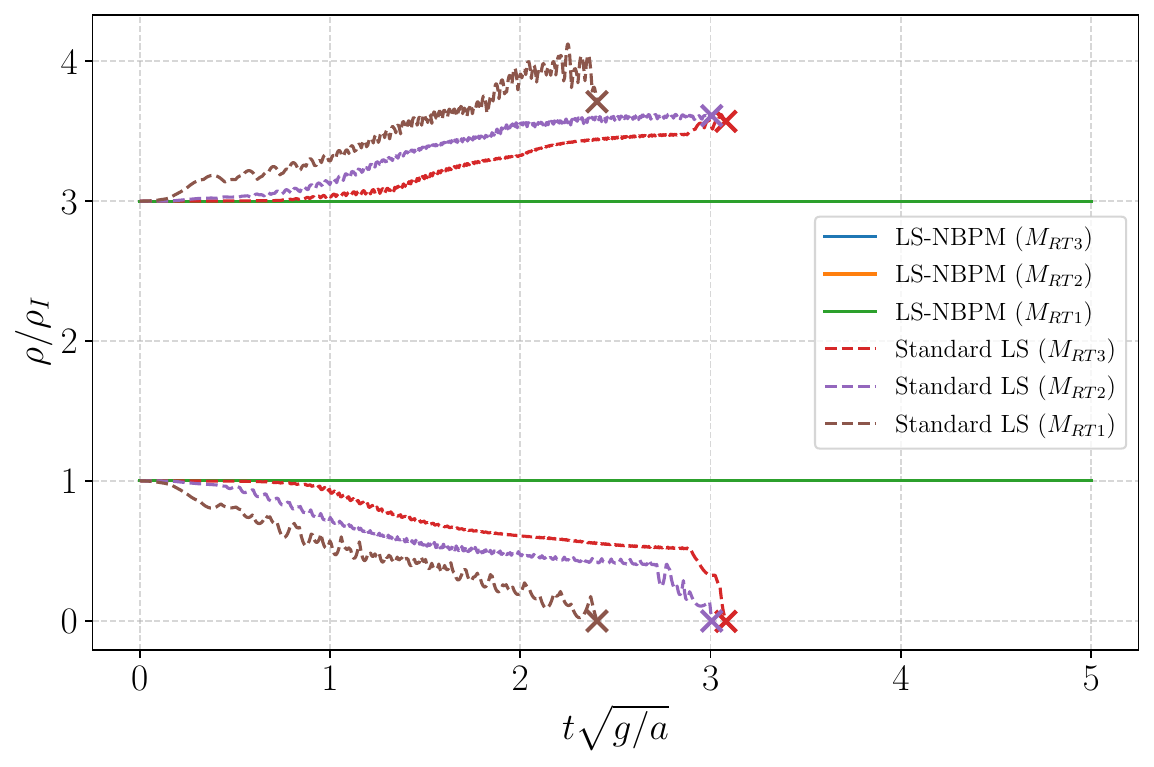}
  \caption{
  Maximum and minimum density values for the LS-NBP and the standard method using  $M_{RT1},\;M_{RT2}, \textrm{and} \;M_{RT3}$ meshes. The crosses indicate the point in time at which the simulation broke down because the density lost positivity. 
  }
  \label{fig:case2_max_min_density_over_time}
\end{figure}

\begin{figure}[htbp]
  \centering
\includegraphics[width=0.8\textwidth]{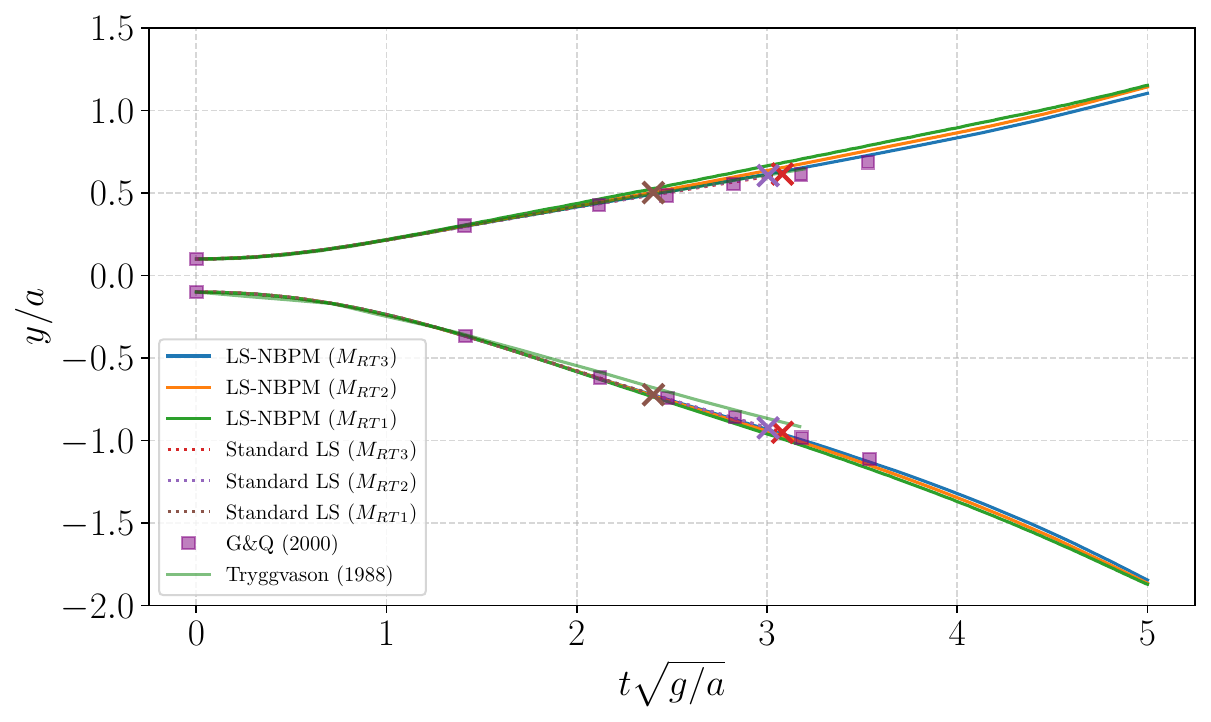}
  \caption{
  Top and bottom bubble height values for the LS-NBP and the standard method using  $M_{RT1},\;M_{RT2}, \textrm{and} \;M_{RT3}$ meshes compared with references. }\label{fig:bubbleHeight}
\end{figure}

\begin{figure}[htbp]
    \centering
    \begin{subfigure}[b]{0.48\textwidth}
        \centering
        \includegraphics[width=\textwidth]{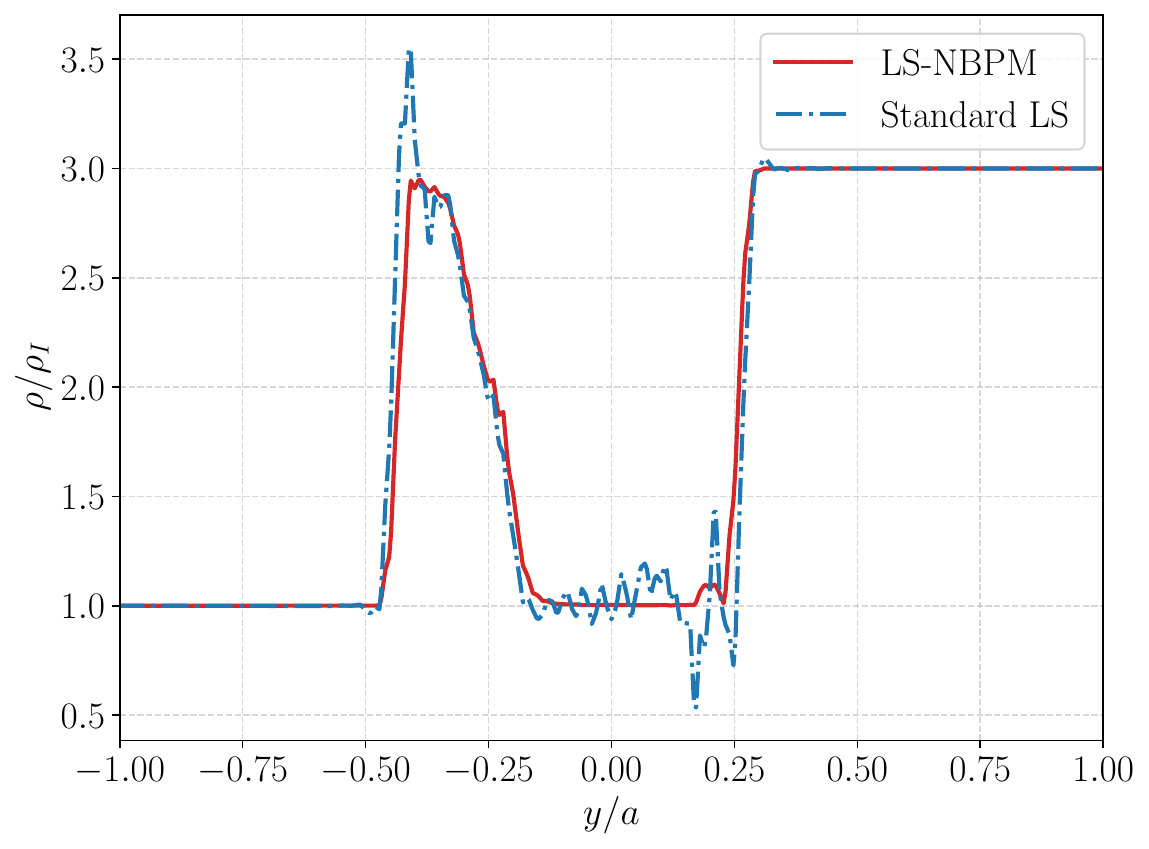}
        \caption{$t\sqrt{g/a}=2$}
        \label{fig:case2_MRT1_density_verticalline_t2}
    \end{subfigure}
    \hfill
    \begin{subfigure}[b]{0.48\textwidth}
        \centering
        \includegraphics[width=\textwidth]{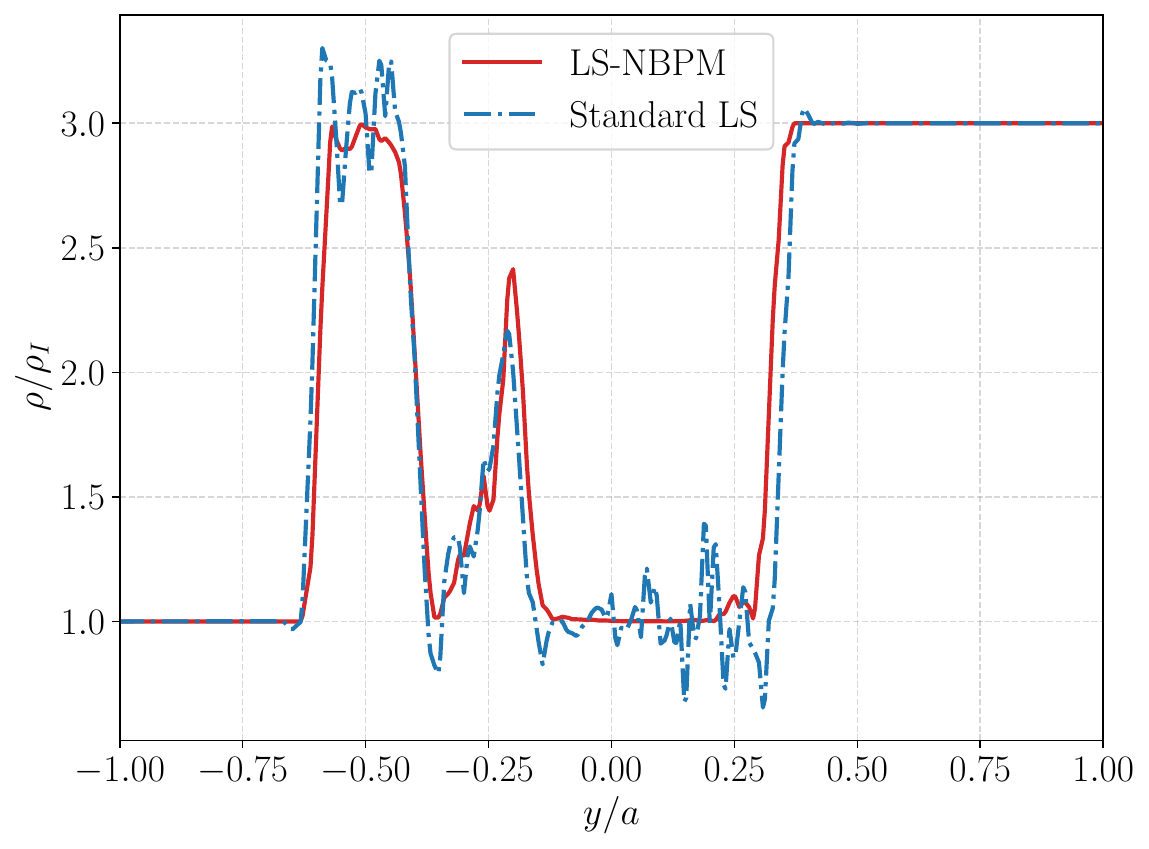}
         \caption{$t\sqrt{g/a}=2.4$}
         \label{fig:case2_MRT1_density_verticalline_t265}
    \end{subfigure}
    \caption{Density values along the vertical line $x/a = 0.25$, at specified times, using the $M_{RT1}$ mesh.}
    \label{fig:case2_MRT1_densityoververtialline_t2and265}
\end{figure}

\begin{figure}[htbp]
    \centering
    \begin{subfigure}[b]{0.48\textwidth}
        \centering
        \includegraphics[width=\textwidth]{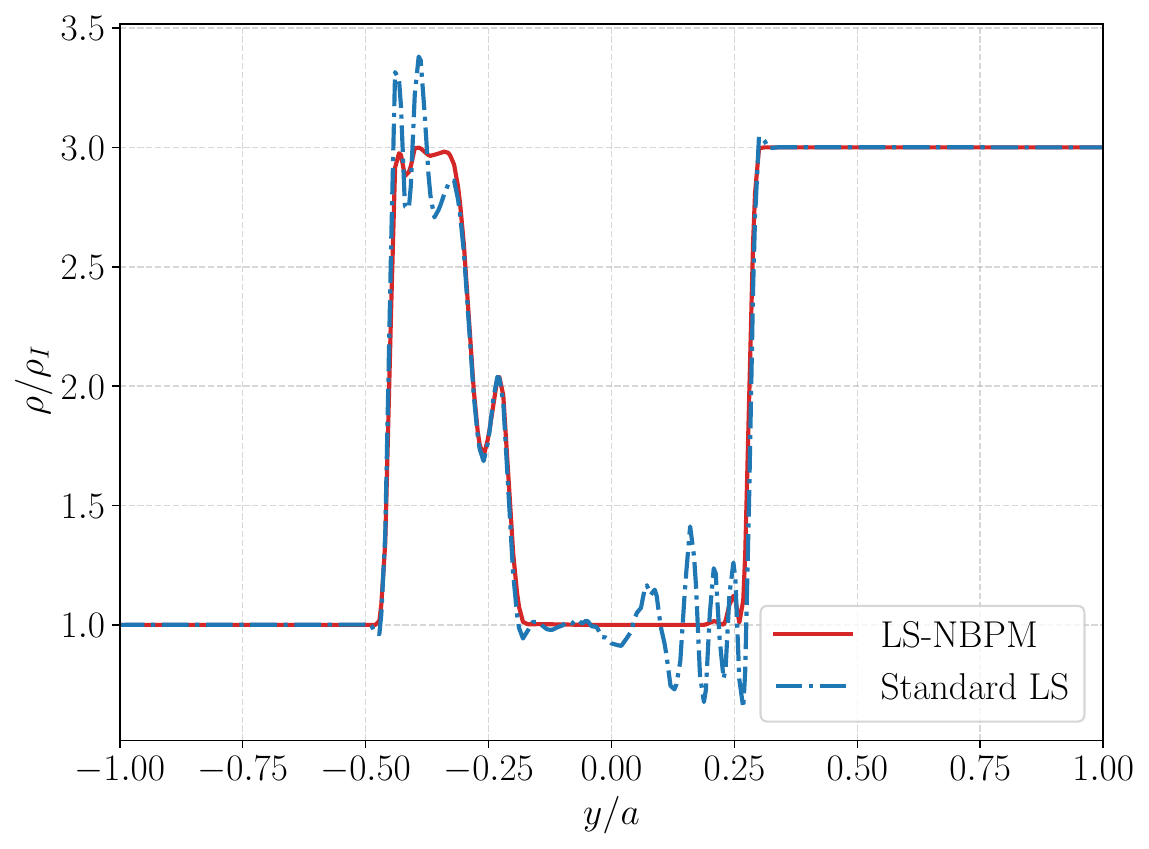}
        \caption{$t\sqrt{g/a}=2$}
        \label{fig:case2_density_verticalline_t2}
    \end{subfigure}
    \hfill
    \begin{subfigure}[b]{0.48\textwidth}
        \centering
        \includegraphics[width=\textwidth]{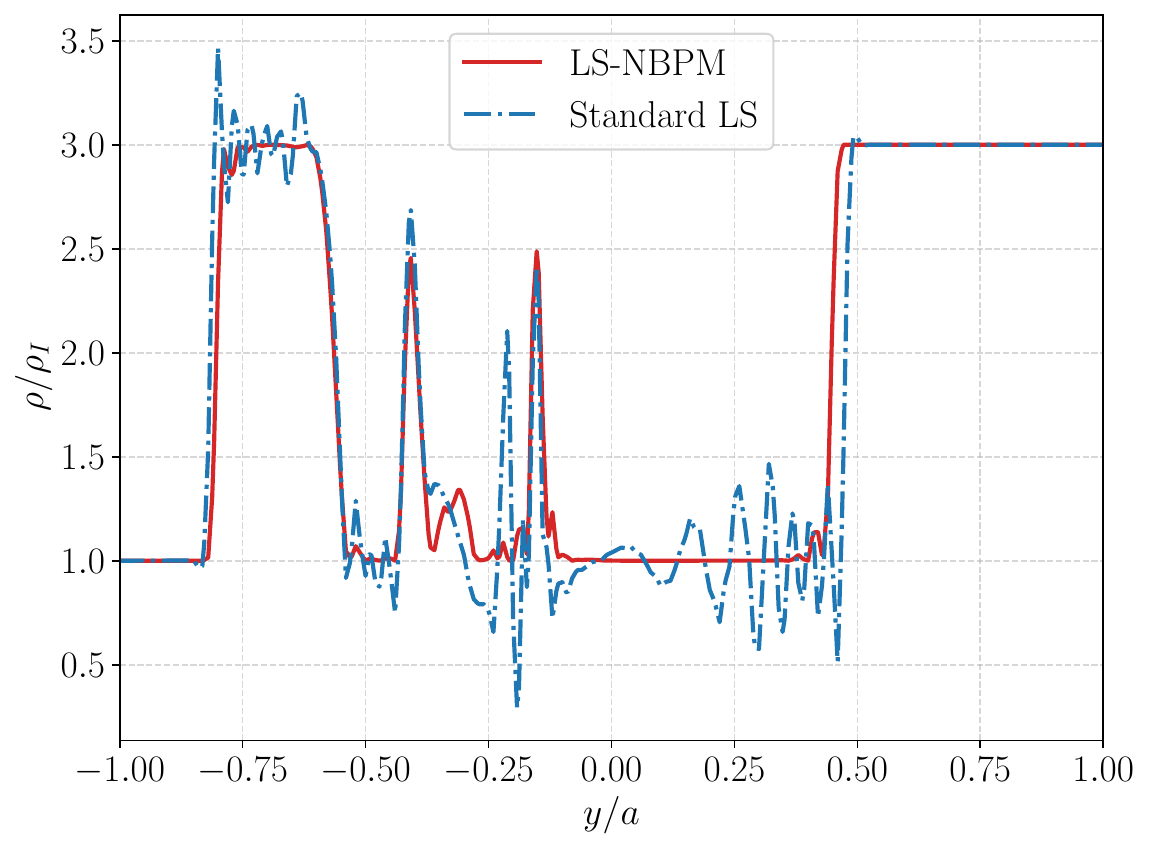}
         \caption{$t\sqrt{g/a}=3$}
         \label{fig:case2_density_verticalline_t3}
    \end{subfigure}
    \caption{Density values along the vertical line $x/a = 0.25$, at specified times, using the $M_{RT2}$ mesh.}
    \label{fig:case2_densityoververtialline_t2and3}
\end{figure}

\begin{figure}[htbp]
  \centering
\includegraphics[width=0.95\textwidth]{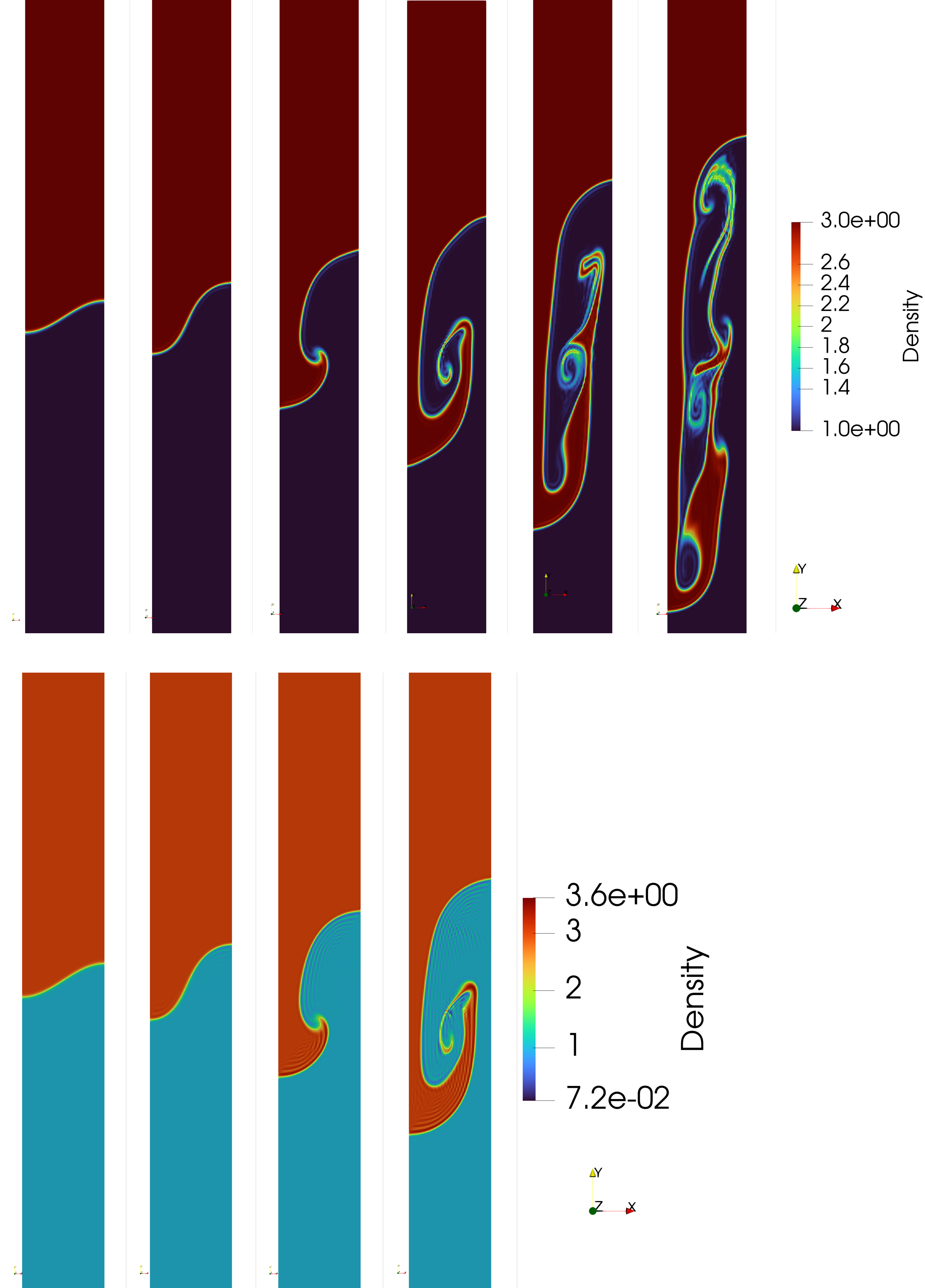}
  \caption{
  Snapshots: LS-NBPM (top) and standard LS (bottom) methods using  $M_{RT2}$ mesh at times $t\sqrt{g/a}=0,\; 1,\; 2,\; 3,\; 4,\; 5$ (top), and times $t\sqrt{g/a}=0,\; 1,\; 2,\; 3$ (bottom). }\label{fig:snapshots}
\end{figure}

\section{Concluding remarks}\label{Sec:Concl}
This work has extended the nodally bound-preserving finite element method to the linear advection equation. We have used a least-squares baseline to yield unconditional temporal stability, especially for the Crank--Nicolson scheme. Additionally, with least-squares we get symmetric algebraic problems at each nonlinear iteration---which enables using efficient iterative solvers, such as preconditioned conjugate gradients. We have proved stability and convergence without any time-step restrictions, whereas most existing second-order bound-preserving schemes induce CFL conditions. Another contribution was using Anderson acceleration to reduce the number of nonlinear iterations. This approach has yielded substantial gains, in many cases over a ten-fold reduction in the overall iteration count. 

We have also tested our method in a more complex scenario where the advection equation is part of a two-phase incompressible flow solver. In that case, bound preservation has proved essential for achieving stable simulations. Our results indicate that, besides preserving extrema, nodally bound-preserving methods have an important stabilising effect, essentially eliminating spurious spatial oscillations.

\section*{Acknowledgements}
This research was funded by the European Union under the Horizon Europe programme, Marie Skłodowska-Curie Actions (Grant Agreement No. 101236419).

\bibliographystyle{abbrv}
\bibliography{refs}

\begin{thebibliography}{10}

\bibitem{AlnaesEtal-2014}
M.~S. Alnaes, A.~Logg, K.~B. {\O}lgaard, M.~E. Rognes, and G.~N. Wells.
\newblock Unified form language: A domain-specific language for weak
  formulations of partial differential equations.
\newblock {\em {ACM} Transactions on Mathematical Software}, 40, 2014.

\bibitem{Amiri-Barrenechea-Pryer-2025}
A.~Amiri, G.~R. Barrenechea, and T.~Pryer.
\newblock A nodally bound-preserving finite element method for time-dependent
  convection–diffusion equations.
\newblock {\em J. Comput. Appl. Math.}, 470:116691, 2025.

\bibitem{ABP-26-M3AS}
A.~Amiri, G.~R. Barrenechea, and T.~Pryer.
\newblock A finite element method preserving the eigenvalue range of symmetric
  tensors.
\newblock {\em Mathematical Models and Methods in Applied Sciences}, 2026.
\newblock In press.

\bibitem{Anderson-1965}
D.~G. Anderson.
\newblock Iterative procedures for nonlinear integral equations.
\newblock {\em J. ACM}, 12(4):547–560, Oct. 1965.

\bibitem{petsc-user-ref-2026}
S.~Balay, S.~Abhyankar, M.~F. Adams, S.~Benson, J.~Brown, P.~Brune,
  K.~Buschelman, E.~M. Constantinescu, L.~Dalcin, A.~Dener, V.~Eijkhout,
  J.~Faibussowitsch, W.~D. Gropp, V.~Hapla, T.~Isaac, P.~Jolivet, D.~Karpeev,
  D.~Kaushik, M.~G. Knepley, F.~Kong, S.~Kruger, D.~A. May, L.~C. McInnes,
  R.~T. Mills, L.~Mitchell, T.~Munson, J.~E. Roman, K.~Rupp, P.~Sanan,
  J.~Sarich, B.~F. Smith, H.~Suh, S.~Zampini, H.~Zhang, H.~Zhang, and J.~Zhang.
\newblock {PETSc/TAO} users manual.
\newblock Technical Report ANL-21/39 - Revision 3.26, Argonne National
  Laboratory, 2026.

\bibitem{BarattaEtal-2023}
I.~A. Baratta, J.~P. Dean, J.~S. Dokken, M.~Habera, J.~S. Hale, C.~N.
  Richardson, M.~E. Rognes, M.~W. Scroggs, N.~Sime, and G.~N. Wells.
\newblock {DOLFINx}: the next generation {FEniCS} problem solving environment.
\newblock preprint, 2023.

\bibitem{BGPV-2024}
G.~R. Barrenechea, E.~H. Georgoulis, T.~Pryer, and A.~Veeser.
\newblock A nodally bound-preserving finite element method.
\newblock {\em IMA Journal of Numerical Analysis}, 44(4):2198--2219, 2024.

\bibitem{BoGu16}
P.~Bochev and M.~Gunzburger.
\newblock Chapter 12 - least-squares methods for hyperbolic problems.
\newblock In R.~Abgrall and C.-W. Shu, editors, {\em Handbook of Numerical
  Methods for Hyperbolic Problems}, volume~17 of {\em Handbook of Numerical
  Analysis}, pages 289--317. Elsevier, 2016.

\bibitem{Cai2021}
W.~Cai, B.~Li, and Y.~Li.
\newblock Error analysis of a fully discrete finite element method for variable
  density incompressible flows in two dimensions.
\newblock {\em ESAIM: Mathematical Modelling and Numerical Analysis},
  55:S103--S147, 2021.

\bibitem{CaJi88}
G.~F. Carey and B.~N. Jiang.
\newblock Least-squares finite elements for first-order hyperbolic systems.
\newblock {\em International Journal for Numerical Methods in Engineering},
  26(1):81--93, 1988.

\bibitem{Chen2020}
H.~Chen, J.~Mao, and J.~Shen.
\newblock {Error estimate of Gauge–Uzawa methods for incompressible flows
  with variable density}.
\newblock {\em Journal of Computational and Applied Mathematics}, 364:112321,
  2020.

\bibitem{dalcinpazklercosimo-2011}
L.~D. Dalcin, R.~R. Paz, P.~A. Kler, and A.~Cosimo.
\newblock {Parallel distributed computing using Python}.
\newblock {\em Advances in Water Resources}, 34(9):1124 -- 1139, 2011.
\newblock New Computational Methods and Software Tools.

\bibitem{Ern-Guermond-I-2021}
A.~Ern and J.-L. Guermond.
\newblock {\em Finite elements I - Approximation and interpolation}, volume~72
  of {\em Texts in Applied Mathematics}.
\newblock Springer, 2021.

\bibitem{Ern-Guermond-III-2021}
A.~Ern and J.-L. Guermond.
\newblock {\em Finite Elements III - First-Order and Time-Dependent PDEs},
  volume~74 of {\em Texts in Applied Mathematics}.
\newblock Springer, 2021.

\bibitem{Espinoza2025}
N.~Espinoza-Contreras, G.~R. Barrenechea, E.~Castillo, and D.~R. Pacheco.
\newblock {Unconditionally stable, linearised IMEX schemes for incompressible
  flows with variable density}.
\newblock {\em ESAIM: Mathematical Modelling and Numerical Analysis},
  59:2739--2761, 2025.

\bibitem{Godunov}
S.~K. Godunov.
\newblock A difference method for numerical calculation of discontinuous
  solutions of the equations of hydrodynamics.
\newblock {\em Mat. Sb. (N. S.)}, 47(89):271--306, 1959.

\bibitem{Guermond2000}
J.-L. Guermond and L.~Quartapelle.
\newblock {A projection FEM for variable density incompressible flows}.
\newblock {\em J. Comput. Phys.}, 165(1):167--188, 2000.

\bibitem{Guermond2009}
J.-L. Guermond and A.~Salgado.
\newblock {A splitting method for incompressible flows with variable density
  based on a pressure Poisson equation}.
\newblock {\em Journal of Computational Physics}, 228:2834--2846, 2009.

\bibitem{GueSal11}
J.-L. Guermond and A.~J. Salgado.
\newblock Error analysis of a fractional time-stepping technique for
  incompressible flows with variable density.
\newblock {\em SIAM Journal on Numerical Analysis}, 49(3):917--944, 2011.

\bibitem{Keith-FCOM}
B.~Keith and T.~M. Surowiec.
\newblock Proximal {G}alerkin: A structure-preserving finite element method for
  pointwise bound constraints.
\newblock {\em Found.~Comput.~Math.}, 26:385–481, 2026.

\bibitem{Stampacchia-2000}
D.~Kinderlehrer and G.~Stampacchia.
\newblock {\em An introduction to variational inequalities and their
  applications}, volume~31 of {\em Classics in applied mathematics}.
\newblock Society for Industrial and Applied Mathematics, 2000.

\bibitem{Kirby-NumMath}
R.~C. Kirby and D.~Shapero.
\newblock High-order bounds-satisfying approximation of partial differential
  equations via finite element variational inequalities.
\newblock {\em Numer. Math.}, 156(3):927--947, 2024.

\bibitem{KuzminH23}
D.~Kuzmin and H.~Hajduk.
\newblock {\em Property-Preserving Numerical Schemes for Conservation Laws}.
\newblock World Scientific, 2023.

\bibitem{LeVeque_2002}
R.~J. LeVeque.
\newblock {\em Finite Volume Methods for Hyperbolic Problems}.
\newblock Cambridge Texts in Applied Mathematics. Cambridge University Press,
  2002.

\bibitem{Pacheco2024}
D.~Pacheco.
\newblock On the temporal stability of least-squares methods for linear
  hyperbolic problems.
\newblock {\em Computers \& Mathematics with Applications}, 168:33--38, 2024.

\bibitem{Pan2024}
J.~Pan and W.~Cai.
\newblock {Error analysis of energy-conservative BDF2–FE scheme for the 2D
  Navier–Stokes equations with variable density}.
\newblock {\em Communications in Nonlinear Science and Numerical Simulation},
  136:108093, 2024.

\bibitem{Saad_2025}
Y.~Saad.
\newblock Acceleration methods for fixed-point iterations.
\newblock {\em Acta Numerica}, 34:805–890, 2025.

\bibitem{Salgado2013}
A.~Salgado.
\newblock A diffuse interface fractional time-stepping technique for
  incompressible two-phase flows with moving contact lines.
\newblock {\em ESAIM: Mathematical Modelling and Numerical Analysis},
  47:743--769, 2013.

\bibitem{BasixJoss-2022}
M.~W. Scroggs, I.~A. Baratta, C.~N. Richardson, and G.~N. Wells.
\newblock Basix: a runtime finite element basis evaluation library.
\newblock {\em Journal of Open Source Software}, 7(73):3982, 2022.

\bibitem{ScroggsEtal-2022}
M.~W. Scroggs, J.~S. Dokken, C.~N. Richardson, and G.~N. Wells.
\newblock Construction of arbitrary order finite element degree-of-freedom maps
  on polygonal and polyhedral cell meshes.
\newblock {\em ACM Transactions on Mathematical Software}, 48(2):{18:1--18:23},
  2022.

\bibitem{Shu2016}
C.-W. Shu.
\newblock {\em Discontinuous Galerkin Methods for Time-Dependent Convection
  Dominated Problems: Basics, Recent Developments and Comparison with Other
  Methods}, pages 371--399.
\newblock Springer International Publishing, Cham, 2016.

\bibitem{Try88}
G.~Tryggvason.
\newblock Numerical simulations of the {R}ayleigh-{T}aylor instability.
\newblock {\em J.~Comput.~Phys.}, 75(2):253--282, 1988.

\bibitem{Walker-Ni-2011}
H.~F. Walker and P.~Ni.
\newblock Anderson acceleration for fixed-point iterations.
\newblock {\em SIAM J. Numer. Anal.}, 49(4):1715–1735, Aug. 2011.

\bibitem{ZS2011}
X.~Zhang and C.-W. Shu.
\newblock Maximum-principle-satisfying and positivity-preserving high-order
  schemes for conservation laws: survey and new developments.
\newblock {\em Proceedings of the Royal Society A: Mathematical, Physical and
  Engineering Sciences}, 467(2134):2752--2776, 2011.

\end{thebibliography}
\end{document}